\documentclass[reqno,11pt]{amsart}

\usepackage{amscd}
\usepackage{amsmath}
\usepackage[new]{old-arrows}
\usepackage{amssymb}
\usepackage[american]{babel}
\usepackage{bbm}
\usepackage{bookmark}
\usepackage{cmap} 
\usepackage{dsfont}
\usepackage{enumerate}
\usepackage{epigraph}
\usepackage[mathscr]{euscript}
\usepackage[myheadings]{fullpage}
\usepackage{graphicx}
\usepackage[geometry]{ifsym}
\usepackage{mathabx}
\usepackage{accents}
\usepackage{mathrsfs}
\usepackage{mathtools}
\usepackage{enumitem}
\usepackage{refcount}
\usepackage{setspace}
\usepackage{stmaryrd}
\usepackage{thinsp}
\usepackage{verbatim}
\usepackage[all,2cell]{xy} \UseTwocells
\usepackage{xr}
\usepackage[multiple]{footmisc}
\usepackage{hyperref}
\usepackage[style=numeric]{biblatex}
\usepackage{tikz}
\usetikzlibrary{decorations.pathmorphing}
\usetikzlibrary{cd}

\numberwithin{equation}{subsection}

\newtheorem{thm}{Theorem}[subsubsection]
\newtheorem*{thm*}{Theorem}

\newtheorem{cor}[thm]{Corollary}
\newtheorem*{cor*}{Corollary}
\newtheorem{lem}[thm]{Lemma}

\newtheorem{prop}[thm]{Proposition}
\newtheorem{prop-const}[thm]{Proposition-Construction}

\newtheorem*{conjecture*}{Conjecture}

\newtheorem*{princ*}{Principle}

\newtheorem{introthm}{Theorem}

\theoremstyle{remark}
\newtheorem{rem}[thm]{Remark}
\newtheorem{example}[thm]{Example}

\newcommand{\into}{\hookrightarrow}

\newcommand{\bB}{{\mathbb B}}

\newcommand{\bG}{{\mathbb G}}

\newcommand{\bQ}{{\mathbb Q}}

\newcommand{\bZ}{{\mathbb Z}}

\newcommand{\cC}{{\mathcal C}}
\newcommand{\cD}{{\mathcal D}}

\newcommand{\cF}{{\mathcal F}}
\newcommand{\cG}{{\mathcal G}}

\newcommand{\cN}{{\mathcal N}}

\newcommand{\cY}{{\mathcal Y}}

\newcommand{\fg}{{\mathfrak g}}
\newcommand{\fh}{{\mathfrak h}}

\newcommand{\ft}{{\mathfrak t}}

\newcommand{\fz}{{\mathfrak z}}

\newcommand{\on}{\operatorname}

\renewcommand{\dot}{\bullet}

\newcommand{\Bun}{\on{Bun}}

\renewcommand{\lim}{\on{lim}}

\renewcommand{\subset}{\subseteq}

\makeatletter
\newcommand{\biggg}{\bBigg@{4}}
\newcommand{\Biggg}{\bBigg@{5}}
\makeatother

\date{\today}

\begin{document}

\frenchspacing

\setlength{\epigraphwidth}{0.4\textwidth}
\renewcommand{\epigraphsize}{\footnotesize}

\begin{abstract}

In this paper, we introduce the category of quasi-tempered automorphic D-modules, which is a rather natural class of D-modules from the point of view of geometric Langlands. We provide a characterization of this category in terms of singular support, and as a consequence, we obtain certain microlocal categorical Künneth formulas.

\end{abstract}

\title{Quasi-Tempered Automorphic D-modules}

\author{Joakim F\ae rgeman}

\address{The University of Texas at Austin, 
Department of Mathematics, 
PMA 11.156, 2515 Speedway Stop C1200, 
Austin, TX 78712}

\email{joakim.faergeman@utexas.edu}

\maketitle

\setcounter{tocdepth}{2}
\tableofcontents

\section{Introduction}

\subsection{Geometric Langlands}

\subsubsection{}

Let $k$ be an algebraically closed field of characteristic zero. We let $G$ be a connected reductive group over $k$ and $X$ a smooth connected projective curve over $k$. Denote by $\textrm{Bun}_G$ the moduli stack of principal $G$-bundles on $X$.

Let $\check{G}$ be the Langlands dual group of $G$, and let $\textrm{LocSys}_{\check{G}}$ denote the derived stack of de Rham $\check{G}$-local systems on $X$.

\subsubsection{} 

Let us begin by recalling the geometric Langlands conjecture as formulated in \cite[\S11]{arinkin2015singular}. Here, one expects an equivalence
\begin{equation}\label{GL}
\mathbb{L}_G: D(\textrm{Bun}_G)\simeq \textrm{IndCoh}_{\check{\mathcal{N}}}(\textrm{LocSys}_{\check{G}})
\end{equation}

\noindent where the right hand side of the equivalence denotes the DG category of ind-coherent sheaves on $\textrm{LocSys}_{\check{G}}$ with nilpotent singular support.\footnote{We refer to \cite{arinkin2015singular} for a discussion of singular support for coherent sheaves.} This category contains the full subcategory $\textrm{QCoh}(\textrm{LocSys}_{\check{G}})$ of quasi-coherent sheaves on $\textrm{LocSys}_{\check{G}}$.

\subsubsection{}

The subcategory $D(\textrm{Bun}_G)^{\textrm{temp}}\subset D(\textrm{Bun}_G)$ of \emph{tempered} D-modules on $\textrm{Bun}_G$, whose definition we recall in \S2.3 below, is supposed to match $\textrm{QCoh}(\textrm{LocSys}_{\check{G}})$ under $\mathbb{L}_G$.

That is, we should have a commutative diagram
\[\begin{tikzcd}
	{D(\textrm{Bun}_G)^{\textrm{temp}}} && {\textrm{QCoh}(\textrm{LocSys}_{\check{G}})} \\
	\\
	{D(\textrm{Bun}_G)} && {\textrm{IndCoh}_{\check{\mathcal{N}}}(\textrm{LocSys}_{\check{G}})}
	\arrow[from=1-1, to=1-3]
	\arrow[hook, from=1-3, to=3-3]
	\arrow[hook, from=1-1, to=3-1]
	\arrow[from=3-1, to=3-3]
\end{tikzcd}\]
where the horizontal functors are equivalences.
\medskip

\subsubsection{} This paper is concerned with an intermediate category $D(\textrm{Bun}_G)^{\textrm{quasi-temp}}$ consisting of \emph{quasi-tempered} D-modules on $\textrm{Bun}_G$ (see §\ref{s:qtemp} below for a definition). If $G$ is non-abelian, this category refines the above picture:
\[\begin{tikzcd}
	{D(\textrm{Bun}_G)^{\textrm{temp}}} && {\textrm{QCoh}(\textrm{LocSys}_{\check{G}})} \\
	& {} \\
	{D(\textrm{Bun}_G)^{\textrm{quasi-temp}}} && {\textrm{IndCoh}_{\check{\mathcal{N}}_{\textrm{Irreg}}}(\textrm{LocSys}_{\check{G}})} \\
	\\
	{D(\textrm{Bun}_G)} && {\textrm{IndCoh}_{\check{\mathcal{N}}}(\textrm{LocSys}_{\check{G}})}
	\arrow[from=1-1, to=1-3]
	\arrow[hook, from=1-3, to=3-3]
	\arrow[hook, from=1-1, to=3-1]
	\arrow[from=3-1, to=3-3]
	\arrow[hook, from=3-1, to=5-1]
	\arrow[from=5-1, to=5-3]
	\arrow[hook, from=3-3, to=5-3]
\end{tikzcd}\]
\\
where all horizontal functors are equivalences. Here, $\textrm{IndCoh}_{\check{\mathcal{N}}_{\textrm{Irreg}}}(\textrm{LocSys}_{\check{G}})$ denotes the category of ind-coherent sheaves on $\textrm{LocSys}_{\check{G}}$ with irregular nilpotent singular support.

\subsection{Main results}

The main goal of this paper is to provide a microlocal characterization of quasi-tempered automorphic D-modules.

\subsubsection{}

Recall that the cotangent bundle $T^*\textrm{Bun}_G$ is the algebraic stack that parametrizes Higgs bundles on $X$. That is (up to identifying $\mathfrak{g}^*\simeq \mathfrak{g}$ by a $G\times \mathbb{G}_m$-invariant isomorphism), $T^*\textrm{Bun}_G$ parametrizes pairs $(\mathcal{P},\phi)$ where $\mathcal{P}$ is a $G$-bundle on $X$ and $\phi\in\Gamma(X,\mathfrak{g}_{\mathcal{P}}\otimes \Omega^1_X)$. Decompose the Lie algebra $\mathfrak{g}=[\mathfrak{g},\mathfrak{g}]\oplus \mathfrak{z}(\mathfrak{g})$ into its derived sub Lie algebra and its center. We shall consider the full subcategory $$D_{\mathfrak{z}(\mathfrak{g})}(\textrm{Bun}_G)\subset D(\textrm{Bun}_G)$$
consisting of elements $\mathcal{F}\in D(\textrm{Bun}_G)$ such that if $(\mathcal{P},\phi)$ lies in the singular support of $\mathcal{F}$, then $\phi$ factors through $\mathfrak{z}(\mathfrak{g})\subset \mathfrak{g}$. Note that if $G$ is semisimple, then $\mathfrak{z}(\mathfrak{g})=0$, and the category $D_{\mathfrak{z}(\mathfrak{g})}(\textrm{Bun}_G)$ coincides with $D_0(\textrm{Bun}_G),$ the category of automorphic local systems.

\subsubsection{} The inclusion
$$D(\textrm{Bun}_G)^{\textrm{quasi-temp}}\longhookrightarrow D(\textrm{Bun}_G)$$

\noindent admits a continuous right adjoint given by \emph{quasi-temperization}: $$\textrm{quasi-temp}:D(\textrm{Bun}_G)\longrightarrow D(\textrm{Bun}_G)^{\textrm{quasi-temp}}$$ 

\subsubsection{} We let $D(\textrm{Bun}_G)^{\textrm{max-anti-temp}}\subset D(\textrm{Bun}_G)$ denote the kernel of quasi-temp. That is, we have an exact sequence of DG categories:
\[\begin{tikzcd}
	{D(\textrm{Bun}_G)^{\textrm{quasi-temp}}} && {D(\textrm{Bun}_G)} && {D(\textrm{Bun}_G)^{\textrm{max-anti-temp}}}.
	\arrow[shift left=1, hook, from=1-1, to=1-3]
	\arrow["{\textrm{max-anti-temp}}", shift left=1, from=1-3, to=1-5]
	\arrow["{\textrm{quasi-temp}}", shift left=1, from=1-3, to=1-1]
	\arrow[shift left=1, hook, from=1-5, to=1-3]
\end{tikzcd}\]
We refer to $D(\textrm{Bun}_G)^{\textrm{max-anti-temp}}$ as the category of \emph{maximally anti-tempered} automorphic D-modules.

\subsubsection{}
We may now state the main theorem in this paper.

\begin{introthm}\label{t:A}

The category $D(\on{Bun}_G)^{\on{max-anti-temp}}$ coincides with $D_{\mathfrak{z}(\mathfrak{g})}(\Bun_G)$.

\end{introthm}

\begin{rem}
One may alternatively define $D(\on{Bun}_G)^{\mathrm{max-anti-temp}}$ in terms of the maximally anti-tempered subcategory $\on{Sph}_G^{\mathrm{max-anti-temp}}$ of the spherical category $\on{Sph}_G$, see §\ref{s:qtemp}. What allows us to control $D(\on{Bun}_G)^{\mathrm{max-anti-temp}}$ is the fact that $\on{Sph}_G^{\mathrm{max-anti-temp}}$ is generated by the dualizing sheaf, see Proposition \ref{p:key}.
\end{rem}

\begin{rem}\label{r:D}

As noted in Example 2.6 below, if $G$ is semisimple of rank 1, then $D(\textrm{Bun}_G)^{\textrm{quasi-temp}}$ coincides with $D(\textrm{Bun}_G)^{\textrm{temp}}$. In this case, Theorem A says that the category $D(\textrm{Bun}_G)^{\textrm{anti-temp}}$ (i.e. the subcategory of $D(\textrm{Bun}_G)$ right orthogonal to $D(\textrm{Bun}_G)^{\textrm{temp}}$) coincides with $D_0(\textrm{Bun}_G)$, the category of automorphic local systems. This was shown by Beraldo in \cite{beraldo2021geometric}.\footnote{There is a slight issue in the statement of Theorem 4.3.8 in \emph{loc.cit}, which yields a marginally different claim. At the time of writing, the paper is being revised.}

\end{rem}

\subsubsection{A mysterious symmetry.}

In \cite{faergeman2022non}, it was shown that 
\[\textrm{Shv}_{Nilp}(\textrm{Bun}_G)^{\textrm{anti-temp}}=\textrm{Shv}_{Nilp_{\textrm{Irreg}}}(\textrm{Bun}_G)
\] 

\noindent as subcategories of $\textrm{Shv}_{Nilp}(\textrm{Bun}_G)$ (see \emph{loc.cit} for the notation).

That is, the anti-tempered automorphic sheaves with nilpotent singular support are exactly those with irregular nilpotent singular support. Under the (restricted) geometric Langlands conjecture considered in \cite{arinkin2020stack}, this is saying that the kernel of the composition
\begin{equation}\label{eq:rgl1}\textrm{Shv}_{Nilp}(\textrm{Bun}_G)\simeq \textrm{IndCoh}_{\check{\mathcal{N}}}(\textrm{LocSys}_{\check{G}}^{\textrm{restr}})\rightarrow \textrm{IndCoh}_0(\textrm{LocSys}_{\check{G}}^{\textrm{restr}})
\end{equation}

\noindent is $\textrm{Shv}_{Nilp_{\textrm{Irreg}}}(\textrm{Bun}_G)$. Here $\textrm{IndCoh}_0(\textrm{LocSys}_{\check{G}}^{\textrm{restr}})$ denotes the category of ind-coherent sheaves on $\textrm{LocSys}_{\check{G}}^{\textrm{restr}}$ with zero singular support and coincides with $\textrm{QCoh}(\textrm{LocSys}_{\check{G}}^{\textrm{restr}})$.

On the other hand, it follows from Theorem A above that
\[
\on{Shv}_{Nilp}(\on{Bun}_G)^{\on{max-anti-temp}}\simeq \on{Shv}_0(\textrm{Bun}_G).
\]

\noindent Under the restricted geometric Langlands conjecture, this is saying that the kernel of the composition 
\begin{equation}\label{eq:rgl2}
\textrm{Shv}_{Nilp}(\textrm{Bun}_G)\simeq \textrm{IndCoh}_{\check{\mathcal{N}}}(\textrm{LocSys}_{\check{G}}^{\textrm{restr}})\rightarrow \textrm{IndCoh}_{\check{\mathcal{N}}_{\textrm{Irreg}}}(\textrm{LocSys}_{\check{G}}^{\textrm{restr}})
\end{equation}

\noindent is given by $\textrm{Shv}_0(\textrm{Bun}_G)$.

Thus, we see a duality appearing between "zero singular support" and "irregular nilpotent singular support" on the automorphic and spectral side of geometric Langlands.
We emphasize that this duality is, to our knowledge, completely mysterious.

\begin{rem}\label{r:lys}
We should note that a preprint of Sergey Lysenko recently appeared establishing a conjectural duality as above for all special nilpotent orbits under the Lusztig-Spaltenstein bijection (see \cite[Conjecture 2.2.2]{lysenko}). The conjecture is partially motivated by the paper \cite[Conjecture 4.2, 4.3]{jiang2014automorphic}, establishing a similar conjecture for automorphic forms.

We remark that the current paper and \cite{faergeman2022non} settle the conjecture of Lysenko for the two extreme cases given by the zero and regular orbit, respectively.
\end{rem}

\subsubsection{Künneth Formulas.}

Theorem A provides a description of the category $D_{\mathfrak{z}(\mathfrak{g})}(\textrm{Bun}_G)$ in terms of the Hecke action on $D(\textrm{Bun}_G)$. As a consequence, this category enjoys various properties not immediate from its definition. Theorem B below provides an example of such a property, showing that $D_{\mathfrak{z}(\mathfrak{g})}(\textrm{Bun}_G)$ satisfies a categorical Künneth formula.

To elaborate, recall that if $\mathcal{X},\mathcal{Y}$ are prestacks locally almost of finite type, then one has a functor 
\[
D(\mathcal{X})\otimes D(\mathcal{Y})\xrightarrow{-\boxtimes -} D(\mathcal{X}\times \mathcal{Y}).
\]
If either $D(\mathcal{X})$ or $D(\mathcal{Y})$ is dualizable as a DG category, then this functor is an equivalence, cf. \cite[Prop. 2.2.8]{drinfeld2013compact}).\footnote{By \emph{loc.cit}, this happens e.g. for $\mathcal{X}=\textrm{Bun}_G$ or if $\mathcal{X}$ is a QCA stack.}

If now $\mathcal{X}, \mathcal{Y}$ are smooth algebraic stacks and $\mathcal{N}_{\mathcal{X}}\subset T^*\mathcal{X}, \mathcal{N}_{\mathcal{Y}}\subset T^*\mathcal{Y}$ are Zariski-closed conical subsets, one similarly has a functor
\begin{equation}\label{eq:sing}
D_{\mathcal{N}_{\mathcal{X}}}(\mathcal{X})\otimes D_{\mathcal{N}_{\mathcal{Y}}}(\mathcal{Y})\xrightarrow{-\boxtimes -} D_{\mathcal{N}_{\mathcal{X}}\times \mathcal{N}_{\mathcal{Y}}}(\mathcal{X}\times \mathcal{Y}).
\end{equation}

\noindent If e.g. $D_{\mathcal{N}_{\mathcal{X}}}(\mathcal{X})$ and $D(\mathcal{Y})$ are dualizable, then the functor (\ref{eq:sing}) is fully faithful. However, the functor need not be essentially surjective in this case. For example, this fails for $\mathcal{X}=\mathcal{Y}=\mathbb{A}^1$ and $\mathcal{N}=\lbrace \textrm{zero section}\rbrace$. Indeed, denote by $$m: \mathbb{A}^1\times \mathbb{A}^1\rightarrow \mathbb{A}^1$$ the multiplication map. Then the D-module $m^!(\textrm{exp})$ does not lie in the image of the functor $$D_0(\mathbb{A}^1)\otimes D_0(\mathbb{A}^1)\xrightarrow{-\boxtimes -} D_0(\mathbb{A}^1\times \mathbb{A}^1)$$
where $\textrm{exp}\in D_0(\mathbb{A}^1)$ denotes the exponential sheaf (\cite[\S 16.3.1]{arinkin2020stack}).

\subsubsection{}

An important class of examples where the functor (\ref{eq:sing}) is an equivalence (at least up to issues of left-completion)\footnote{These issues may often be overcome in the case of interest; for example when $X$ is an abelian variety (see Theorem 4.13 below).}, is when $\mathcal{X}=X$ is a smooth and proper scheme, $\mathcal{N}_X$ is the zero section, and $\mathcal{N}_{\mathcal{Y}}\subset T^*\mathcal{Y}$ is half-dimensional (cf. \cite[Thm A.3.9]{gaitsgory2022toy}).

Our next theorem provides another example where (\ref{eq:sing}) is an equivalence.

\begin{introthm}\label{t:B}
Let $G$ and $H$ be reductive groups. Then the functors 
\[
D_{\mathfrak{z}(\mathfrak{g})}(\on{Bun}_G)\otimes D_{\mathfrak{z}(\mathfrak{h})}(\Bun_H)\longrightarrow D_{\mathfrak{z}(\mathfrak{g})\times \mathfrak{z}(\mathfrak{h})}(\Bun_G\times \Bun_H)
\]
\[D_0(\Bun_G)\otimes D_0(\Bun_H)\longrightarrow D_0(\Bun_G\times \Bun_H)
\]
are equivalences.
\end{introthm}

\subsubsection{}

Note that Theorem B is not very surprising in the light of the aforementioned theorem of \cite{gaitsgory2022toy}. Indeed, if $G$ is semisimple (so that in particular $\mathfrak{z}(\mathfrak{g})=0$), one may write 
\[
\textrm{Bun}_G\simeq \textrm{Maps}(X-x,G)\backslash \textrm{Gr}_{G,x}
\]
where $\textrm{Gr}_{G,x}$ is the affine Grassmannian at $x\in X$, and $\textrm{Maps}(X-x,G)$ denotes the group ind-scheme of maps $\lbrace X-x\rightarrow G\rbrace $.

In this way, $\textrm{Bun}_G$ behaves like a smooth and proper scheme, and so Theorem B is expected. For a general reductive group $G$, one may consider the isogeny $Z(G)^{\circ}\times [G,G]\rightarrow G$ to reduce Theorem B to the case where $G$ is the product of a torus and a semisimple algebraic group. Since for a torus $T$, $\textrm{Bun}_T$ looks like the product of an abelian variety with a lattice (up to a negligible stacky factor), Theorem B is expected in this case as well.

In this paper, however, we will not pursue this line of reasoning. In fact, after proving Theorem A, Theorem B will be easy to deduce.

\subsection{Structure of the paper} In section 2, we review preliminaries on singular support and temperedness in geometric Langlands. In section 3, we study how quasi-temperedness and maximal anti-temperedness behave under functoriality. In section 4, we prove Theorem A and B.

\subsection{Acknowledgements}
I would like to thank Tom Gannon, Rok Gregoric, Sergey Lysenko and Kendric Schefers for helpful discussions related to the present text. Special thanks goes to my advisor, Sam Raskin, for an extensive read of a draft of this paper and for many suggestions that led to radical improvements.

Finally, Dario Beraldo's paper \cite{beraldo2021geometric} has served as a major inspiration for this project.

\section{Preliminaries}

In this section, we review preliminaries needed in the paper.

\subsection{Notation and conventions}

\subsubsection{}

Denote by $k[[t]]$ and $k((t))$ the rings of Taylor series and Laurent series respectively, and write $O:=\textrm{Spec}\: k[[t]]$, $K:=\textrm{Spec}\: k((t))$. For a $k$-valued point $x\in X$, we write $O_x$ and $K_x$ for the adic disc and punctured disc at $x$. Choosing a uniformizer, we have identifications $O_x\simeq O$, $K_x\simeq K$.

\subsection{Sheaves and singular support} Let us set up notation for singular support of D-modules. Our conventions follow those of \cite[Appendix \S E-F]{arinkin2020stack}.

\subsubsection{Stacks} In this paper, by a "smooth algebraic stack" we will mean a smooth stack that is locally QCA in the sense of \cite{drinfeld2013some}. This includes stacks such as $\textrm{Bun}_G$.

\subsubsection{D-modules} For a prestack $\mathcal{Y}$ locally almost of finite type, we denote by $D(\mathcal{Y})$ the DG category of D-modules on $\mathcal{Y}$, following \cite{gaitsgory2017study}. If $f: \mathcal{X}\rightarrow \mathcal{Y}$ is a map of laft prestacks, we have the $!$-pullback functor $f^!: D(\mathcal{Y})\rightarrow D(\mathcal{X})$. Whenever its left adjoint is defined, we denote it by $f_!$. 
Moreover:
\begin{itemize}
    \item If $f$ is ind-representable, we have the $*$-pushforward functor $$f_{*,\textrm{dR}}: D(\mathcal{X})\rightarrow D(\mathcal{Y}).$$
    Whenever its left adjoint is defined, we denote it by $f^{*,\textrm{dR}}$.
    
    \item If $f$ is a map of locally QCA stacks, we let $f_{*,\textrm{ren-dR}}: D(\mathcal{X})\rightarrow D(\mathcal{Y})$ denote the \emph{renormalized} pushforward functor defined in \cite[\S 9]{drinfeld2013some}.
\end{itemize}

\subsubsection{Singular support on schemes} Let $S$ be a smooth scheme of finite type. The category $D(S)$ of D-modules on $S$ carries a natural t-structure, and for each Zariski-closed conical $\mathcal{N}\subset T^*S$, we may consider the category $$D_{\mathcal{N}}(S)^{\heartsuit,c}$$ of coherent D-modules $\mathcal{F}\in D(S)^{\heartsuit, c}$ whose singular support is contained in $\mathcal{N}$.

Write $D_{\mathcal{N}}(S)^{\heartsuit}:=\textrm{Ind}(D_{\mathcal{N}}(S)^{\heartsuit,c})$ for its Ind-completion. We denote by $D_{\mathcal{N}}(S)$ the subcategory of $D(S)$ consisting of D-modules $\mathcal{F}$ such that each cohomology sheaf $H^n(\mathcal{F})$ is an element of $D_{\mathcal{N}}(S)^{\heartsuit}$. By construction, $D_{\mathcal{N}}(S)$ carries a t-structure making the inclusion $D_{\mathcal{N}}(S)\hookrightarrow D(S)$ t-exact. Moreover, $D_{\mathcal{N}}(S)$ is left-complete in this t-structure.

\subsubsection{Singular support on stacks.}

Let now $\mathcal{Y}$ be an arbitrary smooth algebraic stack, and let $\mathcal{N}\subset T^*\mathcal{Y}$ be a Zariski-closed conical subset. For each finite type scheme $S$ mapping smoothly to $\mathcal{Y}$, denote by $\mathcal{N}_S$ the image of the codifferential $$T^*\mathcal{Y}\underset{\mathcal{Y}}{\times} S\rightarrow T^*S.$$

\noindent We define $D_{\mathcal{N}}(\mathcal{Y})$ to be $$D_{\mathcal{N}}(\mathcal{Y}):=\underset{S}{\textrm{lim}}\: D_{\mathcal{N}_S}(S)$$
where the limit is taken over finite type schemes mapping smoothly to $\mathcal{Y}$. That is, $D_{\mathcal{N}}(\mathcal{Y})$ is the full subcategory of $D(\mathcal{Y})$ consisting of sheaves that pull back to $D_{\mathcal{N}_S}(S)$ for each finite type scheme $S$ mapping smoothly to $\mathcal{Y}$. We see that $D_{\mathcal{N}}(\mathcal{Y})$ carries a natural structure, which is also left-complete.

In this paper, we will often consider the case where $\mathcal{N}=\lbrace \textrm{zero section}\rbrace$. We denote the corresponding category by $D_0(\mathcal{Y})$. If $\mathcal{Y}=S$ is a scheme, $D_0(S)$ consists of those $\mathcal{F}\in D(S)$ such that for every $n\in \mathbb{Z}$, each coherent subsheaf $\mathcal{G}\subset H^n(\mathcal{F})$ is a vector bundle with a flat connection.

\subsubsection{}

By construction, we have:
\begin{lem}\label{l:!}
Let $f:\mathcal{X}\rightarrow \mathcal{Y}$ be a smooth map of smooth algebraic stacks. For a Zariski-closed conical subset $\mathcal{N}_{\mathcal{Y}}\subset T^*\mathcal{Y}$, denote by $\mathcal{N}_{\mathcal{X}}$ the image of the codifferential
\[
df^{\dot}: T^*{\mathcal{Y}}\underset{\mathcal{Y}}{\times}\mathcal{X}\rightarrow T^*{\mathcal{X}}.
\]
Then the functor $f^!:D(\mathcal{Y})\rightarrow D(\mathcal{X})$ restricts to a functor 
\[
f^!:D_{\mathcal{N}_{\mathcal{Y}}}(\mathcal{Y})\rightarrow D_{\mathcal{N}_{\mathcal{X}}}(\mathcal{X}).
\]
Moreover, if $f$ is surjective and $\mathcal{F}\in D(\mathcal{Y})$ satisfies $f^!(\mathcal{F})\in D_{\mathcal{N}_{\mathcal{X}}}(\mathcal{X})$, then $\mathcal{F}\in D_{\mathcal{N}_{\mathcal{Y}}}(\mathcal{Y})$.
\end{lem}

\begin{proof}
First, recall that $f$ is t-exact up to a cohomological shift by the relative dimension of $f$ (viewed as a locally constant function on $\mathcal{Y}$). If $\mathcal{X},\mathcal{Y}$ are finite type smooth schemes, the lemma follows from the fact that if $\cF\in D(\mathcal{Y})^{\heartsuit, c}$, then the singular support of $f^!(\cF)$ is the image of the singular support of $\cF$ under the codifferential $df^{\bullet}$, and the latter is an embedding.

The general case now reduces to that of finite type smooth schemes, by definition.

\end{proof}

\begin{lem}\label{l:*} 
Let $f: \mathcal{X}\rightarrow \mathcal{Y}$ be a schematic proper map of smooth algebraic stacks. Let $q$ denote the projection
\[
T^*\mathcal{Y}\underset{\mathcal{Y}}{\times}\mathcal{X}\rightarrow T^*\mathcal{Y}.
\]
For a Zariski-closed conical subset $\mathcal{N}_{\mathcal{X}}\subset T^*\mathcal{X}$, write $\mathcal{N}_{\mathcal{Y}}:=q\big((df^{\dot})^{-1}(\mathcal{N}_{\mathcal{X}})\big)$. Then $f_{*,\textrm{dR}}: D(\mathcal{X})\rightarrow D(\mathcal{Y})$ restricts to a functor
\[
f_{*,\textrm{dR}}: D_{\mathcal{N}_{\mathcal{X}}}(\mathcal{X})\rightarrow D_{\mathcal{N}_{\mathcal{Y}}}(\mathcal{Y}).
\]
\end{lem}
\begin{proof}
By base-change, we may assume that $\mathcal{X}$ and $\mathcal{Y}$ are smooth schemes. In this case, the assertion follows from the classical fact that for each $n\in\mathbb{Z}$, the functor $$H^n(f_{*,\textrm{dR}}(-)): D_{\mathcal{N}_{\mathcal{X}}}(\mathcal{X})^{\heartsuit,c}\rightarrow D(\mathcal{Y})^{\heartsuit}$$
lands in $D_{\mathcal{N}_{\mathcal{Y}}}(\mathcal{Y})^{\heartsuit,c}$ whenever $f$ is proper.
\end{proof}

\subsubsection{Singular support for automorphic sheaves.} We will be primarily concerned with singular support of sheaves on $\textrm{Bun}_G$. As such, recall that the cotangent bundle, $T^*\textrm{Bun}_G$, parametrizes pairs $(\mathcal{P},\phi)$ where $\mathcal{P}$ is a $G$-bundle on $X$ and $\phi\in \Gamma(X,\mathfrak{g}_{\mathcal{P}}\otimes\Omega_X^1)$. Whenever $\Lambda$ is an Ad-invariant Zariski-closed conical subset of $\mathfrak{g}$, we may consider the subset $(T^*\textrm{Bun}_G)_{\Lambda}\subset T^*\textrm{Bun}_G$ consisting of pairs $(\mathcal{P},\phi)\in T^*\textrm{Bun}_G$ such that $\phi$ factors through $\Lambda$. Note that this is a Zariski-closed conical subset of $T^*\textrm{Bun}_G$. 

We denote by $$D_{\Lambda}(\textrm{Bun}_G)\subset D(\textrm{Bun}_G)$$
the corresponding subcategory of sheaves $\mathcal{F}\in D(\textrm{Bun}_G)$
whose singular support is contained in $(T^*\textrm{Bun}_G)_{\Lambda}$. 

Recall that given $x\in X$, we have an action of the \emph{spherical category} (see $\S 2.3.1$ below) on $D(\textrm{Bun}_G)$ given by Hecke modifications at $x$. The argument in \cite{nadler2019spectral} shows that for any $\Lambda$, this action preserves $D_{\Lambda}(\textrm{Bun}_G)$.\footnote{In \emph{loc.cit}, the assertion is only explicitly stated for $\Lambda$ being the nilpotent cone of $\mathfrak{g}$. However, the same argument applies verbatim for any Zariski-closed Ad-invariant conical $\Lambda\subset \mathfrak{g}$.}

\begin{rem}
In this paper, we often consider the cases when $\Lambda=\mathfrak{z}(\mathfrak{g})$ or $\Lambda=0$. As we will see, the categories $D_{\mathfrak{z}(\mathfrak{g})}(\textrm{Bun}_G)$ and $D_0(\textrm{Bun}_G)$ admit characterizations purely in terms of the Hecke action on $D(\textrm{Bun}_G)$.\footnote{For $D_{\mathfrak{z}(\mathfrak{g})}(\textrm{Bun}_G)$, this will follow from Theorem A. For $D_0(\textrm{Bun}_G)$, this follows from Theorem A and \cite[]{nadler2019spectral} Namely, a sheaf $\mathcal{F}\in D(\textrm{Bun}_G)$ lies in $D_0(\textrm{Bun}_G)$ if and only if it lies in $D_{\mathfrak{z}(\mathfrak{g})}(\textrm{Bun}_G)$ and is locally constant along the curve for the Hecke action (see \emph{loc.cit}).} Recently, other examples of this phenomenon have occurred: see \cite{nadler2019spectral}, \cite{arinkin2020stack} and \cite{faergeman2022non} for the cases of $\Lambda$ being the locus of nilpotent elements and irregular elements, respectively.
\end{rem} 

\subsection{Temperedness} Before introducing the notion of quasi-temperedness, let us recall the usual notion of temperedness in geometric Langlands.
\subsubsection{} Consider the monoidal category of bi-$G(O)$-equivariant D-modules on $G(K)$: $$\textrm{Sph}_G:=D\big(G(O)\backslash G(K)/G(O)\big).$$
We write $\Omega \mathfrak{\check{g}}:=\textrm{pt}\underset{\mathfrak{\check{g}}}{\times}\textrm{pt}.$ As described in \cite{bezrukavnikov2007equivariant}, \cite[\S 12]{arinkin2015singular}, we have an equivalence of monoidal categories 
\begin{equation}\label{eq:sat}
\textrm{Sat}_G: \textrm{Sph}_G\simeq \textrm{IndCoh}_{\textrm{Nilp}(\mathfrak{\check{g}})/\check{G}}(\Omega \mathfrak{\check{g}}/\check{G}).
\end{equation}

\noindent Here, $\textrm{Nilp}(\mathfrak{\check{g}}^*)$ denotes the nilpotent cone of $\mathfrak{\check{g}}^*$, and the monoidal structure on the right hand side is given by convolution. $\textrm{IndCoh}_{\textrm{Nilp}(\mathfrak{\check{g}})/\check{G}}(\Omega \mathfrak{\check{g}}/\check{G})$ contains $\textrm{QCoh}(\Omega \mathfrak{\check{g}}/\check{G})$, the category of quasi-coherent sheaves on $\Omega\mathfrak{\check{g}}/\check{G}$, as a full subcategory. We define the category $\textrm{Sph}_G^{\textrm{temp}}$ to be the inverse image of $\textrm{QCoh}(\Omega \mathfrak{\check{g}}/\check{G})$ under the equivalence (\ref{eq:sat}).

There is an adjunction pair\[\begin{tikzcd}
	{\textrm{QCoh}(\Omega\mathfrak{\check{g}}/\check{G})} && {\textrm{IndCoh}_{\textrm{Nilp}(\mathfrak{\check{g}^*})/\check{G}}(\Omega \mathfrak{\check{g}}/\check{G})}
	\arrow[shift left=1, hook, from=1-1, to=1-3]
	\arrow[shift left=1, from=1-3, to=1-1]
\end{tikzcd}\]

\noindent from which we see that the inclusion $$\textrm{Sph}_G^{\textrm{temp}}\longhookrightarrow \textrm{Sph}_G$$ admits a continuous right adjoint $$\textrm{temp}: \textrm{Sph}_G\longrightarrow \textrm{Sph}_G^{\textrm{temp}}.$$

We define $\textrm{Sph}_G^{\textrm{anti-temp}}$ to be the kernel of $\textrm{temp}$. Thus, we obtain an exact sequence of DG categories 
\[\begin{tikzcd}
	{\textrm{Sph}_G^{\textrm{temp}}} && {\textrm{Sph}_G} && {\textrm{Sph}_G^{\textrm{anti-temp}}}
	\arrow[shift left=1, hook, from=1-1, to=1-3]
	\arrow["{\textrm{temp}}", shift left=1, from=1-3, to=1-1]
	\arrow[shift left=1, hook, from=1-5, to=1-3]
	\arrow["{\textrm{anti-temp}}", shift left=1, from=1-3, to=1-5]
\end{tikzcd}\]
where the map $\textrm{anti-temp}: \textrm{Sph}_G\rightarrow \textrm{Sph}_G^{\textrm{anti-temp}}$ is given by $\mathcal{G}\mapsto \textrm{cone}(\textrm{temp}(\mathcal{G})\rightarrow \mathcal{G})$.

\subsubsection{} For any DG category $\mathcal{C}$ equipped with an action of $\textrm{Sph}_G$, we define $$C^{\textrm{temp}}:=\textrm{Sph}_G^{\textrm{temp}}\underset{\textrm{Sph}_G}{\otimes}\mathcal{C},\:\: \mathcal{C}^{\textrm{anti-temp}}:=\textrm{Sph}_G^{\textrm{anti-temp}}\underset{\textrm{Sph}_G}{\otimes}\mathcal{C}.$$
As before, we have an exact sequence of DG categories
\begin{equation}\label{eq:temp}
\begin{tikzcd}
	{\mathcal{C}^{\textrm{temp}}} && {\mathcal{C}} && {\mathcal{C}^{\textrm{anti-temp}}}.
	\arrow[shift left=1, hook, from=1-1, to=1-3]
	\arrow["{\textrm{temp}}", shift left=1, from=1-3, to=1-1]
	\arrow[shift left=1, hook, from=1-5, to=1-3]
	\arrow["{\textrm{anti-temp}}", shift left=1, from=1-3, to=1-5]
\end{tikzcd}
\end{equation}

\subsubsection{Quasi-temperedness.}\label{s:qtemp} Let us now consider a slight variant of the above construction. Namely, consider the category $$\textrm{Sph}_G^{\textrm{quasi-temp}}:= \textrm{Sat}_G^{-1}\big(\textrm{IndCoh}_{\textrm{Nilp}_{\textrm{Irreg}}(\mathfrak{\check{g}})/\check{G}}(\Omega\mathfrak{\check{g}}/\check{G})\big)$$
where $\textrm{Nilp}_{\textrm{Irreg}}(\check{\mathfrak{g}}^*)$ denotes the locus of irregular nilpotent elements of $\check{\mathfrak{g}}^*$. We will refer to this category as the \emph{quasi-tempered spherical category}. We define $\textrm{Sph}_G^{\textrm{max-anti-temp}}$ by the exact sequence
\[\begin{tikzcd}
	{\textrm{Sph}_G^{\textrm{quasi-temp}}} && {\textrm{Sph}_G} && {\textrm{Sph}_G^{\textrm{max-anti-temp}}}.
	\arrow[shift left=1, hook, from=1-1, to=1-3]
	\arrow["{\textrm{quasi-temp}}", shift left=1, from=1-3, to=1-1]
	\arrow[shift left=1, hook, from=1-5, to=1-3]
	\arrow["{\textrm{max-anti-temp}}", shift left=1, from=1-3, to=1-5]
\end{tikzcd}\]
Similarly, for any DG category $\mathcal{C}$ equipped with an action of $\textrm{Sph}_G$, we define
\begin{equation}\label{eq:qtemp}
C^{\textrm{quasi-temp}}:=\textrm{Sph}_G^{\textrm{quasi-temp}}\underset{\textrm{Sph}_G}{\otimes}\mathcal{C},\:\: \mathcal{C}^{\textrm{max-anti-temp}}:=\textrm{Sph}_G^{\textrm{max-anti-temp}}\underset{\textrm{Sph}_G}{\otimes}\mathcal{C}.
\end{equation}
In particular, we get a "quasi-tempered" version of \ref{eq:temp}:
\[\begin{tikzcd}
	{\mathcal{C}^{\textrm{quasi-temp}}} && {\mathcal{C}} && {\mathcal{C}^{\textrm{max-anti-temp}}}.
	\arrow[shift left=1, hook, from=1-1, to=1-3]
	\arrow["{\textrm{quasi-temp}}", shift left=1, from=1-3, to=1-1]
	\arrow[shift left=1, hook, from=1-5, to=1-3]
	\arrow["{\textrm{max-anti-temp}}", shift left=1, from=1-3, to=1-5]
\end{tikzcd}\]
We refer to elements of $\mathcal{C}^{\textrm{max-anti-temp}}$ as \emph{maximally anti-tempered}.

\subsubsection{} Denote by $\mathbbm{1}_{\textrm{Sph}_G}$ the monoidal unit of $\textrm{Sph}_G$. The tempered unit and quasi-tempered unit are given by $\mathbbm{1}_{\textrm{Sph}_G}^{\textrm{temp}}:=\textrm{temp}(\mathbbm{1}_{\textrm{Sph}_G})$ and $\mathbbm{1}_{\textrm{Sph}_G}^{\textrm{quasi-temp}}:=\textrm{quasi-temp}(\mathbbm{1}_{\textrm{Sph}_G})$, respectively. We do not know of an explicit description of $\mathbbm{1}_{\textrm{Sph}_G}^{\textrm{quasi-temp}}$ outside the case of Example \ref{e:ssrank1} below.

\subsubsection{} Observe that for a DG category $\mathcal{C}$ equipped with an action of $(\textrm{Sph}_G,\star)$, we have identifications of functors 
$$\mathbbm{1}_{\textrm{Sph}_G}^{\textrm{temp}}\star -\simeq \textrm{temp}:\mathcal{C}\rightarrow \mathcal{C}$$
$$\mathbbm{1}_{\textrm{Sph}_G}^{\textrm{quasi-temp}}\star -\simeq \textrm{quasi-temp}:\mathcal{C}\rightarrow \mathcal{C}.$$
Indeed, it suffices to show this for $\cC\simeq \on{Sph}_G$ where it follows from the fact that $\on{temp}$ (resp. $\on{quasi-temp}$) is $\on{Sph}_G$-linear.

In particular, we have $$\mathcal{C}^{\textrm{anti-temp}}=\lbrace c\in \mathcal{C}\vert \mathbbm{1}_{\textrm{Sph}_G}^{\textrm{temp}}\star c=0\rbrace$$
$$\mathcal{C}^{\textrm{max-anti-temp}}=\lbrace c\in \mathcal{C}\vert \mathbbm{1}_{\textrm{Sph}_G}^{\textrm{quasi-temp}}\star c=0\rbrace.$$

\begin{example}

If $G=T$ is a torus, one has $\textrm{Nilp}(\mathfrak{\check{t}})=0$. Thus $\textrm{Sph}_T^{\textrm{temp}}=\textrm{Sph}_T, \: \textrm{Sph}_T^{\textrm{anti-temp}}=0$. On the other hand, $\textrm{Sph}_T^{\textrm{quasi-temp}}=0, \: \textrm{Sph}_T^{\textrm{max-anti-temp}}=\textrm{Sph}_T$. Note that Theorem A is trivial in this case.

\end{example}

\begin{example}\label{e:ssrank1}

If $G$ is of semisimple rank 1, then $0\in\textrm{Nilp}(\mathfrak{\check{g}})$ is the only irregular element. Thus, $\textrm{Sph}_G^{\textrm{quasi-temp}}=\textrm{Sph}_G^{\textrm{temp}}$ and $\textrm{Sph}_G^{\textrm{max-anti-temp}}=\textrm{Sph}_G^{\textrm{anti-temp}}$. In this case, an explicit description of $\mathbbm{1}_{\textrm{Sph}_G}^{\textrm{quasi-temp}}=\mathbbm{1}_{\textrm{Sph}_G}^{\textrm{temp}}$ is provided in \cite[Thm. C]{beraldo2021tempered}.

\end{example}

\subsubsection{} 

For any $x\in X$, we may identify $$D(G(O_x)\backslash G(K_x)/G(O_x))\simeq \textrm{Sph}_G,$$ and thereby obtain an action of $\textrm{Sph}_G$ on $D(\textrm{Bun}_G)$ given by Hecke modifications at $x$. By \cite{faergeman2021arinkin}, the corresponding categories $D(\textrm{Bun}_G)^{\textrm{temp}}, D(\textrm{Bun}_G)^{\textrm{anti-temp}}$, $D(\textrm{Bun}_G)^{\textrm{quasi-temp}}, D(\textrm{Bun}_G)^{\textrm{max-anti-temp}}$ are all independent of the choice of $x$.\footnote{To be precise, the fact that the categories $D(\textrm{Bun}_G)^{\textrm{temp}}, D(\textrm{Bun}_G)^{\textrm{anti-temp}}$ are independent of $x$ is explicitly stated in Theorem 1.1.3.1 in \emph{loc.cit}. For the categories $D(\textrm{Bun}_G)^{\textrm{quasi-temp}}, D(\textrm{Bun}_G)^{\textrm{max-anti-temp}}$, the assertion follows from the combination of Theorem 1.1.3.1 in \emph{loc.cit} and Remark 3.5 below which says that $D(\textrm{Bun}_G)^{\textrm{quasi-temp}}$ is the kernel of the functor $\omega_{\textrm{Sph}_G}\star -: D(\textrm{Bun}_G)\rightarrow D(\textrm{Bun}_G)$. Indeed, in the language of \emph{loc.cit}, it is easy to see that $\omega_{\textrm{Sph}_{G,X}}$ is quasi-ULA over $X$. Alternatively, the independence of $x$ follows a posteriori from Theorem A.} Henceforth, we fix $x\in X$ and consider the Hecke action of $\textrm{Sph}_G$ on $D(\textrm{Bun}_G)$ at $x$.
\subsubsection{} As alluded to in the introduction, let us verify that quasi-tempered automorphic sheaves correspond to ind-coherent sheaves with irregular nilpotent singular support on the spectral side of the geometric Langlands equivalence:

\begin{lem}\label{l:gl}

$D(\on{Bun}_G)^{\on{quasi-temp}}$ corresponds to the category $\on{IndCoh}_{\check{\mathcal{N}}_{\on{Irreg}}}(\on{LocSys}_{\check{G}})$ under the (conjectural) geometric Langlands equivalence.

\end{lem}
\begin{proof}
The equivalence $\mathbb{L}_G$ is supposed to be compatible with Hecke actions. That is, $\mathbb{L}_G$ should interchange the action of $\textrm{Sph}_G$ on $D(\textrm{Bun}_G)$ with the action of $\textrm{IndCoh}_{\textrm{Nilp}(\mathfrak{\check{g}}^*)/\check{G}}(\Omega\mathfrak{\check{g}}/\check{G})$ on $\textrm{IndCoh}_{\mathcal{\check{N}}}(\textrm{LocSys}_{\check{G}})$ through the derived geometric Satake equivalence (\ref{eq:sat}) (see \cite[\S12]{arinkin2015singular} for the definition of the spectral action.)

Thus, we need to show that $$\textrm{IndCoh}_{\textrm{Nilp}_{\textrm{Irreg}}(\mathfrak{\check{g})}/\check{G}}(\Omega\mathfrak{\check{g}}/\check{G})\underset{\textrm{IndCoh}_{\textrm{Nilp}(\mathfrak{\check{g}})/\check{G}}(\Omega\mathfrak{\check{g}}/\check{G})}{\otimes}\textrm{IndCoh}_{\mathcal{\check{N}}}(\textrm{LocSys}_{\check{G}})$$ coincides with $$\textrm{IndCoh}_{\check{\mathcal{N}}_{\textrm{Irreg}}}(\textrm{LocSys}_{\check{G}})$$ as full subcategories of $\textrm{IndCoh}_{\mathcal{\check{N}}}(\textrm{LocSys}_{\check{G}})$. Said in another way, we need to show that the action of $\textrm{IndCoh}_{\textrm{Nilp}_{\textrm{Irreg}}(\mathfrak{\check{g}})/\check{G}}(\Omega\mathfrak{\check{g}}/\check{G})$ on $\textrm{IndCoh}_{\mathcal{\check{N}}}(\textrm{LocSys}_{\check{G}})$ lands in $\textrm{IndCoh}_{\mathcal{\check{N}}_{\textrm{Irreg}}}(\textrm{LocSys}_{\check{G}})$ and generates the target under colimits. This, in turn, follows immediately from \cite[Proposition 9.42, Corollary 9.4.3 (b)]{arinkin2015singular}.

More precisely, with the notation in \emph{loc.cit}, we take $\mathcal{X}=\textrm{pt}/\check{G},\: \mathcal{V}=\mathfrak{\check{g}}/\check{G}, \:\mathcal{Z}=\textrm{LocSys}_{\check{G}},\: \mathcal{U}=\textrm{LocSys}_{\check{G}}^{\textrm{R.S.}}, Y=\textrm{Nilp}_{\textrm{Irreg}}(\mathfrak{\check{g}}^*/\check{G})\underset{\textrm{pt}/\check{G}}{\times}\textrm{LocSys}_{\check{G}}^{\textrm{R.S.}}$. Noting that $Y\cap \textrm{Sing}(\textrm{LocSys}_{\check{G}})=\mathcal{\check{N}}_{\textrm{Irreg}}$, Corollary 9.4.3 (b) in \emph{loc.cit} says that the action functor
$$\textrm{IndCoh}_{\textrm{Nilp}_{\textrm{Irreg}}(\mathfrak{\check{g})}/\check{G}}(\Omega\mathfrak{\check{g}}/\check{G}) \underset{\textrm{QCoh}(\textrm{pt}/\check{G})}{\otimes} \textrm{IndCoh}(\textrm{LocSys}_{\check{G}})$$
$$\simeq \bigg(\textrm{IndCoh}(\Omega\mathfrak{\check{g}}/\check{G}))\underset{\textrm{QCoh}(\textrm{pt}/\check{G})}{\otimes}\textrm{IndCoh}(\textrm{LocSys}_{\check{G}})\bigg)_Y$$
\vspace{2mm}
$$\longrightarrow \textrm{IndCoh}_{\check{\mathcal{N}}_{\textrm{Irreg}}}(\textrm{LocSys}_{\check{G}})$$ generates the target under colimits.
\end{proof}

\section{Maximally Anti-Tempered Objects}

In this section, we study the category $\textrm{Sph}_G^{\textrm{max-anti-temp}}$.
\subsection{Key proposition}
The goal of this subsection is to prove the following result:
\begin{prop}\label{p:key}
The category $\on{Sph}_G^{\on{max-anti-temp}}$ is generated by the dualizing sheaf $\omega_{\on{Sph}_G}$ under colimits.
\end{prop}

\subsubsection{}

Before doing so, let us recall another description of the spherical category $\textrm{Sph}_G$ coming from Koszul duality. Namely, by \cite{bezrukavnikov2007equivariant}, \cite[\S12]{arinkin2015singular}, \cite[Lemma 4.3.1]{beraldo2021tempered}, we have a monoidal equivalence
\begin{equation}
\textrm{IndCoh}\big((\mathfrak{\check{g}}/\check{G})^{\wedge}_{\check{\mathcal{N}}/\check{G}}\big)^{\Rightarrow}\simeq \textrm{IndCoh}_{\textrm{Nilp}(\mathfrak{\check{g}}^*)/\check{G}}(\Omega\mathfrak{\check{g}}/\check{G})
\end{equation}

\noindent translating set-theoretic on the LHS to singular support on the RHS. We refer to \cite[\S A.2]{arinkin2015singular} and \cite[\S 2.1]{beraldo2020spectral} for a discussion of the \emph{shift of grading operation} "$\Rightarrow$" occurring above.

In particular, we obtain an equivalence
\begin{equation}
\textrm{IndCoh}\big((\mathfrak{\check{g}}/\check{G})^{\wedge}_{\check{\mathcal{N}}_{\textrm{Irreg}}/\check{G}}\big)^{\Rightarrow}\simeq \textrm{IndCoh}_{\textrm{Nilp}_{\textrm{Irreg}}(\mathfrak{\check{g}}^*)/\check{G}}(\Omega\mathfrak{\check{g}}/\check{G}).
\end{equation}

\subsubsection{} Fix a regular nilpotent element $f\in \mathcal{\check{N}}_{\textrm{reg}}$ and extend it to an $\textrm{sl}_2$-triple $(e,h,f)$. Denote by $\mathfrak{\check{g}}^e$ the centralizer of $e$ in $\mathfrak{\check{g}}$ and consider the Kostant slice $\textrm{Kos}:=f+\mathfrak{\check{g}}^e$. The following important proposition is due to Beraldo \cite[Prop. 4.1.11]{beraldo2021geometric}.

\begin{prop}\label{p:B}
Denote by $\kappa$ the canonical map $\on{Kos}^{\wedge}_f\rightarrow (\mathfrak{\check{g}}/\check{G})^{\wedge}_{\check{\mathcal{N}}/\check{G}}$. Under the equivalences $$\on{Sph}_G\simeq \on{IndCoh}_{\on{Nilp}(\mathfrak{\check{g}}^*)/\check{G}}(\Omega\mathfrak{\check{g}}/\check{G})\simeq \on{IndCoh}\big((\mathfrak{\check{g}}/\check{G})^{\wedge}_{\check{\mathcal{N}}/\check{G}}\big)^{\Rightarrow},$$
the dualizing sheaf $\omega_{\on{Sph}_G}$ corresponds to
$\kappa^{\on{IndCoh}}_*(\omega_{\on{Kos}_f^{\wedge}})\in \on{IndCoh}\big((\mathfrak{\check{g}}/\check{G})^{\wedge}_{\check{\mathcal{N}}/\check{G}}\big)^{\Rightarrow}.$
\end{prop}

\subsubsection{}

Observe that
$\kappa_*^{\textrm{IndCoh}}(\omega_{\textrm{Kos}_f^{\wedge}})$ actually lands in $\textrm{IndCoh}\big((\mathfrak{\check{g}}/\check{G})^{\wedge}_{\check{\mathcal{N}}_{\textrm{reg}}/\check{G}}\big)^{\Rightarrow}$. In fact, we have:

\begin{lem} \label{l:gen}
The category $\on{IndCoh}\big((\mathfrak{\check{g}}/\check{G})^{\wedge}_{\check{\mathcal{N}}_{\on{reg}}/\check{G}}\big)^{\Rightarrow}$ is generated under colimits by the element $\kappa_*^{\on{IndCoh}}(\omega_{\on{Kos}_f^{\wedge}})$.
\end{lem}

\begin{proof}

The shift of grading operation does not play a role in this proof,\footnote{In fact, from the properties described in \cite[\S A.2]{arinkin2015singular} and \cite[\S 2.1]{beraldo2020spectral}, it is easy to see that the versions of Lemma \ref{l:gen} with and without the shift of grading operation imply each other.} so we will omit it in the notation.

Consider the commutative diagram
\[\begin{tikzcd}
	{\lbrace f\rbrace} && {\textrm{Kos}^{\wedge}_f} \\
	\\
	{\mathcal{\check{N}}_{\textrm{reg}}/\check{G}} && {(\mathfrak{\check{g}}/\check{G})^{\wedge}_{\mathcal{\check{N}}_{\textrm{reg}}/\check{G}}}
	\arrow["i", from=1-1, to=1-3]
	\arrow["\pi"', from=1-1, to=3-1]
	\arrow["\kappa", from=1-3, to=3-3]
	\arrow["g", from=3-1, to=3-3]
\end{tikzcd}\]
\emph{Step 1.} Let us start by proving that $\textrm{IndCoh}(\textrm{Kos}_f^{\wedge})$ is generated by $\omega_{\textrm{Kos}_f^{\wedge}}$ under colimits. By \cite[Prop. 3.1.2.]{gaitsgory2017study}, $\textrm{IndCoh}(\textrm{Kos}_f^{\wedge})$ is generated by $i_*^{\textrm{IndCoh}}(k)$, and so it suffices to show that $i_*^{\textrm{IndCoh}}(k)$ can be written as a filtered colimit of $\omega_{\textrm{Kos}_f^{\wedge}}$. Since
\[\textrm{IndCoh}(\textrm{Kos}_f^{\wedge})= \textrm{IndCoh}(\lbrace f\rbrace \underset{\textrm{Kos}_{\textrm{dR}}}{\times}\textrm{Kos})\simeq  \textrm{Vect}\underset{D(\textrm{Kos})}{\otimes}\textrm{IndCoh}(\textrm{Kos})
\]

\noindent we just have to show that $\textrm{IndCoh}(\textrm{Kos})$ is generated under colimits by $\omega_{\textrm{Kos}}$. However, this follows from the fact that $\textrm{Kos}\simeq \mathbb{A}^n$ is an affine space. Indeed, we have an equivalence 
\[\textrm{QCoh}(\textrm{Kos})\rightarrow \textrm{IndCoh}(\textrm{Kos}),\: \mathcal{F}\mapsto \mathcal{F}\overset{\textrm{act}}{\otimes}\omega_{\textrm{Kos}}.
\]

\noindent Here $-\overset{\textrm{act}}{\otimes}-:\textrm{QCoh}(\textrm{Kos})\otimes \textrm{IndCoh}(\textrm{Kos})\rightarrow \textrm{IndCoh}(\textrm{Kos})$ denotes the action of $\textrm{QCoh}(\textrm{Kos})$ on $\textrm{IndCoh}(\textrm{Kos})$ (see \cite[\S1.4]{gaitsgory2011ind}). Since $\textrm{QCoh}(\textrm{Kos})$ is generated by the structure sheaf, this proves the assertion.\\

\emph{Step 2.} By Step 1, we need to show that $\kappa_*^{\textrm{IndCoh}}$ generates the target under colimits. Since the functor $g_*^{\textrm{IndCoh}}:\textrm{IndCoh}(\mathcal{\check{N}}_{\textrm{reg}}/\check{G})\rightarrow \textrm{IndCoh}((\mathfrak{\check{g}}/\check{G})_{\mathcal{\check{N}_{\textrm{reg}}}/\check{G}}^{\wedge})$ generates the target under colimits (cf. \cite[Prop. 3.1.2.]{gaitsgory2017study}), it suffices to show that the same is true for $\pi_*^{\textrm{IndCoh}}: \textrm{IndCoh}(\lbrace f\rbrace )\rightarrow \textrm{IndCoh}(\mathcal{\check{N}}_{\textrm{reg}}/\check{G})$. Write $H$ for the stabilizer of $f$ in $\check{G}$ so that $\textrm{IndCoh}(\mathcal{\check{N}}_{\textrm{reg}}/\check{G})\simeq \textrm{IndCoh}(\textrm{pt}/H)\simeq \textrm{Rep}(H)$. Under this equivalence, the sheaf $\pi_*^{\textrm{IndCoh}}(k)$ goes to the regular represention $k[H]$ of $H$. Thus the assertion follows from the lemma below.
\end{proof}

\begin{lem}\label{l:cobar}

For any affine algebraic group $H$, the regular representation generates $\on{Rep}(H)$ under colimits.
\end{lem}

\begin{proof}

Consider the category $\on{QCoh}(H)$ equipped with its monoidal structure given by convolution. Let $\pi: H\to \on{pt}$ be the projection map. The monoidal functor
\[
\on{QCoh}(H)\xrightarrow{\pi_*} \on{Vect}
\]

\noindent induces an action $\on{QCoh}(H)\curvearrowright \on{Vect}$. We have an equivalence (see e.g. \cite[§2.3]{beraldo2017loop}):
\[
\on{Rep}(H)\simeq \on{Vect}\underset{\on{QCoh}(H)}{\otimes} \on{Vect.}
\]

\noindent In particular, the image of the functor
\[
\on{Vect}\simeq \on{Vect}\underset{\on{QCoh}(H)}{\otimes} \on{QCoh}(H)\xrightarrow{\on{id}\otimes \pi_*}\on{Vect}\underset{\on{QCoh}(H)}{\otimes} \on{Vect}\simeq \on{Rep}(H)
\]

\noindent generates the target under colimits. By construction, the image of $k\in\on{Vect}$ goes to the regular representation.

\end{proof}

\emph{Proof of Proposition \ref{p:key}}. By definition, under derived Satake and Koszul duality, $\textrm{Sph}_G^{\textrm{max-anti-temp}}$ corresponds to $\textrm{IndCoh}\big((\mathfrak{\check{g}}/\check{G})^{\wedge}_{\check{\mathcal{N}}_{\textrm{reg}}/\check{G}}\big)^{\Rightarrow}$. Thus Proposition \ref{p:key} follows from the combination of Proposition \ref{p:B} and Lemma \ref{l:gen}.
\qed

\begin{rem}\label{r:kernel}
As a consequence of Proposition \ref{p:key}, we see that if $\mathcal{C}$ is a DG category equipped with an action of $\textrm{Sph}_G$, then $$\mathcal{C}^{\textrm{quasi-temp}}=\lbrace c\in \mathcal{C}\vert\: \omega_{\textrm{Sph}_G}\star c=0\rbrace.$$

\end{rem}

\subsection{Functoriality}

We now study how quasi-temperedness and maximal anti-temperedness behave under functoriality. First, we draw a couple of corollaries from the previous subsection.
\subsubsection{}

Let $\mathcal{C}$ be a DG category equipped with an action of the loop group $G(K)$. This formally implies that we have an action of $\textrm{Sph}_G$ on the category $\mathcal{C}^{G(O)}$ of $G(O)$-invariant objects of $\mathcal{C}$. By Proposition \ref{p:key}, $\mathcal{C}^{G(O),\textrm{max-anti-temp}}$ is generated under colimits by the essential image of the map $$\omega_{\textrm{Sph}_G}\star -: \mathcal{C}^{G(O)}\rightarrow \mathcal{C}^{G(O)}.$$

This gives an alternative description of $\mathcal{C}^{G(O),\textrm{max-anti-temp}}$:

\begin{cor}\label{c:mon}

$\mathcal{C}^{G(O),\on{max-anti-temp}}$ consists of the $G(K)$-monodromic objects of $\mathcal{C}^{G(O)}$ (i.e. the category generated under colimits by the essential image of the forgetful functor $\mathcal{C}^{G(K)}\rightarrow \mathcal{C}^{G(O)}$.)
\end{cor}

\begin{proof}
By definition, the action of $\omega_{\textrm{Sph}_G}$ on $D(G(O)\backslash G(K))$ is given by $!$-pullback and $!$-pushforward along the correspondence
\[\begin{tikzcd}
	{G(O)\backslash G(K)\times^{G(O)}G(K)} && {G(O)\backslash G(K)} \\
	\\
	{G(O)\backslash G(K)}
	\arrow["q"', from=1-1, to=3-1]
	\arrow["{\textrm{mult}}", from=1-1, to=1-3]
\end{tikzcd}\]
where $q$ is the projection onto the second factor. Denote by $p$ the projection $G(O)\backslash G(K)\rightarrow \textrm{pt}$. Note that $\textrm{mult}_!\circ q^!\simeq p^!\circ p_!$ by base-change. Moreover, we have an isomorphism of endofunctors $$p^!\circ p_!\simeq \textrm{Oblv}^{G(K)\rightarrow G(O)}\circ \textrm{Av}_!^{G(O)\rightarrow G(K)}: D(G(O)\backslash G(K))\rightarrow D(G(K)\backslash G(K))\rightarrow D(G(O)\backslash G(K)).$$

This proves the assertion for the universal case of $\mathcal{C}=D(G(K))$. For a general $\mathcal{C}$, the corollary follows from writing $$\mathcal{C}^{G(O)}\simeq D(G(O)\backslash G(K))\underset{D(G(K))}{\otimes}\mathcal{C}$$ compatibly with the action of $\textrm{Sph}_G$. This isomorphisms holds because invariants and coinvariants agree for $D(G(O))$ (cf. \cite[Thm. 3.6.1]{beraldo2017loop} by noting that $G_1\subset G(O)$ is pro-unipotent and $G(O)/G_1=G$).
\end{proof}

\subsubsection{Automorphic sheaves.} 

Let us return to automorphic sheaves. Consider the stack $\textrm{Bun}_G^{\infty\cdot x}$ parameterizing $G$-bundles on $X$ equipped with a trivialization on the formal disk around $x$. Then $\textrm{Bun}_G^{\infty\cdot x}$ is a $G(O_x)$-torsor over $\textrm{Bun}_G$. Recall that $D(\textrm{Bun}_G^{\infty\cdot x})$ carries an action of $G(K_x)$ by modifying the $G$-bundle around $x$. Since $D(\textrm{Bun}_G)\simeq D(\textrm{Bun}_G^{\infty\cdot x})^{G(O_x)}$ (and identifying $\mathcal{O}_x\simeq \mathcal{O}, \mathcal{K}_x\simeq K$), we may apply the above formalism.

Consider the Hecke stack $\textrm{Hecke}_x^G$ parametrizing triples $(\mathcal{P}_1,\mathcal{P}_2,\tau)$ where $\mathcal{P}_1,\mathcal{P}_2$ are $G$-bundles on $X$ and $\tau$ is an isomorphism of the two bundles on $X-x$. In other words, $\textrm{Hecke}_x^G$ is given by the fiber product $\textrm{Bun}_G\underset{{\textrm{Bun}_G(X-x)}}{\times}\textrm{Bun}_G$. We have two evident projections maps
\[\begin{tikzcd}
	&& {\textrm{Hecke}_x^G} \\
	{\textrm{Bun}_G} &&&& {\textrm{Bun}_G}
	\arrow["{\overleftarrow{h}}"', from=1-3, to=2-1]
	\arrow["{\overrightarrow{h}}", from=1-3, to=2-5]
\end{tikzcd}\]
where $\overleftarrow{h}$ remembers $\mathcal{P}_1$ and $\overrightarrow{h}$ remembers $\mathcal{P}_2$, respectively.

The Hecke stack forms a groupoid over $\textrm{Bun}_G$, and we may consider the category $D(\textrm{Bun}_G)^{\textrm{Hecke}_x^G}$ of $\textrm{Hecke}_x^G$-equivariant D-modules on $\textrm{Bun}_G$. By definition, $D(\textrm{Bun}_G)^{\textrm{Hecke}_x^G}$ is given by the totalization of the cosimplicial category
\[\begin{tikzcd}
	{D(\textrm{Bun}_G)} && {D(\textrm{Hecke}_x^G)} && {\cdot \cdot \cdot}
	\arrow[shift left=1, from=1-1, to=1-3]
	\arrow[shift right=1, from=1-1, to=1-3]
	\arrow[shift left=1, from=1-3, to=1-5]
	\arrow[shift right=1, from=1-3, to=1-5]
	\arrow[from=1-3, to=1-5]
\end{tikzcd}\]

By construction, we have a forgetful functor $D(\textrm{Bun}_G)^{\textrm{Hecke}_x^G}\xrightarrow{\textrm{oblv}^{\textrm{Hecke}_x^G}} D(\textrm{Bun}_G)$. Since the maps $\overleftarrow{h},\overrightarrow{h}$ are ind-proper, we get \cite[Prop 4.4.5.]{gaitsgory2013contractibility}: 

\begin{prop}\label{p:G}
The functor $\on{oblv}^{\on{Hecke}_x^G}$ admits a left adjoint $\on{ind}^{\on{Hecke}_x^G},$ and the monad $\on{oblv}^{\on{Hecke}_x^G}\circ \on{ind}^{\textrm{Hecke}_x^G}$ is isomorphic to $\overleftarrow{h}_!\circ \overrightarrow{h}^!$.
\end{prop}

\subsubsection{} 

Observe that we have an isomorphism of endofunctors
\[
\overleftarrow{h}_!\circ \overrightarrow{h}^!\simeq \omega_{\textrm{Sph}_G}\star -:D(\textrm{Bun}_G)\rightarrow D(\textrm{Bun}_G).
\] 

\noindent From Proposition \ref{p:key}, we obtain:
\begin{cor}\label{c:hecke}

The category $D(\on{Bun}_G)^{\on{max-anti-temp}}$ coincides with the category of $\on{Hecke}_x^G$-monodromic objects of $D(\on{Bun}_G)$ (i.e. the category generated under colimits by the essential image of the forgetful functor $D(\on{Bun}_G)^{\on{Hecke}_x^G}\xrightarrow{\on{oblv}^{\on{Hecke}_x^G}} D(\on{Bun}_G))$.

\end{cor}

\subsubsection{} We we will now see that maximal anti-temperedness plays well under functoriality. To distinguish the spherical actions corresponding to different groups, we write $\textrm{quasi-G-temp}$ when referring to quasi-temperedness on $D(\textrm{Bun}_G)$ instead of $\textrm{quasi-temp}$ etc.

\begin{prop}\label{p:funct}

(a) Let $G\rightarrow H$ be a map of reductive algebraic groups, and suppose $\cC, \cD$ are DG categories equipped with an action of $H(K), G(K)$, respectively. Considering the induced action of $G(K)$ on $\cC$, let $F: \cC\to \cD$ be a $G(K)$-equivariant functor. Then $F$ restricts to a functor  $$F:\cC^{H(O),\on{max-anti-H-temp}}\to \cD^{G(O),\on{max-anti-G-temp}}$$
(b) Assume moreover that $\cC$ and $\cD$ are dualizable as DG-categories. Considering $\cC\otimes \cD$ as a $H(K)\times G(K)$-category in the obvious way, we have a canonical equivalence $$(\cC\otimes \cD)^{H(O)\times G(O),\on{max-anti-G}\times \on{H-temp}}\simeq \cC^{H(O),\on{max-anti-G-temp}}\otimes \cD^{G(O),\on{max-anti-H-temp}}$$
as subcategories of $\cC^{H(O)}\otimes \cD^{G(O)}$.

\end{prop}

\begin{proof}

Note that we may factor the functor 
$$\cC^{H(K)}\xrightarrow{\on{oblv}} \cC\xrightarrow{F} \cD$$
as the composition
$$\cC^{H(K)}\to \cC^{G(K)}\xrightarrow{F} \cD^{G(K)}\xrightarrow{\on{oblv}}\cD.$$

\noindent Assertion (a) now follows from Corollary \ref{c:mon}.

To prove (b), observe that under the equivalence $(\cC\otimes \cD)^{H(O)\times G(O)}\simeq \cC^{H(O)}\otimes \cD^{G(O)}$, the functor $$\omega_{\textrm{Sph}_{G\times H}}\star -: (\cC\otimes \cD)^{H(O)\times G(O)}\rightarrow (\cC\otimes \cD)^{H(O)\times G(O)}$$ coincides with the functor $$(\omega_{\textrm{Sph}_G}\star -)\otimes (\omega_{\textrm{Sph}_H}\star -): \cC^{H(O)}\otimes \cD^{G(O)}\rightarrow \cC^{H(O)}\otimes \cD^{G(O)}.$$

\noindent We are now done by Proposition \ref{p:key}.
\end{proof}

\begin{rem}\label{r:product}

A similar assertion to Proposition \ref{p:funct} (b) holds if we replace maximal anti-temperedness by usual temperedness (essentially by definition).

Another way to see both assertions is to use the Koszul dual description of $\textrm{Sph}_{G\times H}$. Namely, observing that a pair $(x,y)\in \mathfrak{\check{g}}\times \mathfrak{\check{h}}$ is zero (resp. regular) if and only if both $x,y$ are zero (resp. regular), it follows that $$\textrm{IndCoh}\big((\Omega(\mathfrak{\check{g}}\times\mathfrak{\check{h}})/\check{G}\times\check{H})^{\wedge}_{0/\check{G}\times\check{H}}\big)^{\Rightarrow}\simeq \textrm{IndCoh}\big((\Omega\mathfrak{\check{g}}/\check{G})^{\wedge}_{0/\check{G}}\big)^{\Rightarrow}\otimes \textrm{IndCoh}\big((\Omega\mathfrak{\check{h}}/\check{H})^{\wedge}_{0/\check{H}}\big)^{\Rightarrow}$$
and
$$\textrm{IndCoh}\big((\Omega(\mathfrak{\check{g}}\times\mathfrak{\check{h}})/\check{G}\times\check{H})^{\wedge}_{\mathcal{N}_{G\times H,\textrm{reg}}/\check{G}\times\check{H}}\big)^{\Rightarrow} \simeq \textrm{IndCoh}\big((\Omega\mathfrak{\check{g}}/\check{G})^{\wedge}_{\mathcal{N}_{G,\textrm{reg}}/\check{G}}\big)^{\Rightarrow}\otimes \textrm{IndCoh}\big((\Omega\mathfrak{\check{h}}/\check{H})^{\wedge}_{\mathcal{N}_{H,\textrm{reg}}/\check{H}}\big)^{\Rightarrow}.$$

\noindent This covers the universal case from which the assertions follow.

Let us note one more observation for later use. Namely, if $H=T$ is a torus in the formulation of Proposition \ref{p:funct} (b), then the irregular elements of $\check{\ft}\times \check{\fg}$ are exactly the elements $(x,y)$ where $y\in \check{\fg}$ is irregular. It follows that $$(\cC\otimes \cD)^{T(O)\times G(O), \on{quasi-T}\times \on{G-temp}}\simeq \cC^{T(O)}\otimes \cD^{G(O), \on{quasi-G-temp}}.$$

\end{rem}
\subsubsection{} When proving Theorem A, we will do so by reducing to the case where our group $G$ is semisimple and simply-connected. Thus, we will often consider functoriality induced by central isogenies. The following lemma is key in this regard:
\begin{lem}\label{l:key}

Let $G\rightarrow H$ be a central isogeny of reductive groups. Consider the induced maps $f: \on{Bun}_G\rightarrow \on{Bun}_H$, $f_{\on{Sph}}: \on{Sph}_G\rightarrow \on{Sph}_H$. Let $\mathcal{F}\in D(\on{Bun}_H)$, $\mathcal{H}\in \on{Sph}_H$. Then $f_{\on{Sph}}^!(\mathcal{H})\overset{\on{Sph}_G}{\star} f^!(\mathcal{F})$ occurs as a direct summand of $f^!(\mathcal{H}\overset{\on{Sph}_H}{\star}\mathcal{F})$. In particular, if $\mathcal{F}$ is quasi-tempered, then so is $f^!(\mathcal{F})$.

\end{lem}
\begin{proof}
Consider the commutative diagram
\[\begin{tikzcd}
	&& {\textrm{Hecke}_x^G} \\
	&& {\textrm{Bun}_H\underset{\textrm{Bun}_H(X-x)}{\times}\textrm{Bun}_G} \\
	{\textrm{Bun}_G} &&&& {\textrm{Bun}_G} \\
	&& {\textrm{Hecke}_x^H} \\
	{\textrm{Bun}_H} &&&& {\textrm{Bun}_H} \\
	&& {H(O_x)\backslash H(K_x)/H(O_x)}
	\arrow["f"', from=3-1, to=5-1]
	\arrow["{\overleftarrow{h}_H}"', from=4-3, to=5-1]
	\arrow["{\overrightarrow{h}_H}", from=4-3, to=5-5]
	\arrow["f", from=3-5, to=5-5]
	\arrow["{\overleftarrow{h}_G}"', bend right = 20, from=1-3, to=3-1]
	\arrow["{p_r}", from=2-3, to=3-5]
	\arrow["{p_l}", from=2-3, to=4-3]
	\arrow["j", from=1-3, to=2-3]
	\arrow["{\overrightarrow{h}_G}", bend left = 20, from=1-3, to=3-5]
	\arrow["{q_H}", from=4-3, to=6-3]
\end{tikzcd}\]
where the right square is Cartesian. By definition, we have $$f^!(\mathcal{H}\overset{\textrm{Sph}_H}{\star}\mathcal{F})=f^!\circ (\overrightarrow{h}_H)_!(\overleftarrow{h}_H^!(\mathcal{F})\overset{!}{\otimes} q_H^!(\mathcal{H}))$$ which by base change becomes\footnote{Note that we may apply base change here because the maps $\overleftarrow{h}_H, \overrightarrow{h}_H$ are ind-proper.} 
\begin{equation}\label{eq:bigshf}
(p_r)_!\circ p_l^!(\overleftarrow{h}_H^!(\mathcal{F})\overset{!}{\otimes} q_H^!(\mathcal{H}))\simeq (p_r)_!\big((\overleftarrow{h}_H\circ p_l)^!(\mathcal{F})\overset{!}{\otimes} (q_H\circ p_l)^!(\mathcal{H})\big).
\end{equation}

The stack $\textrm{Bun}_H\underset{{\textrm{Bun}_H(X-x)}}{\times}\textrm{Bun}_G$ parametrizes triples $(\mathcal{P}_H,\mathcal{P}_G,\tau)$ where $\mathcal{P}_H$ is an $H$-bundle on $X$, $\mathcal{P}$ is a $G$-bundle, and $\tau$ is an isomorphism $\mathcal{P}_H\vert_{{X-x}}\xrightarrow{\simeq}\mathcal{P}_G\overset{G}{\times} H\vert_{X-x}$. We may write this stack as $$\textrm{Bun}_H\underset{{\textrm{Bun}_H(X-x)}}{\times}\textrm{Bun}_G\simeq \textrm{Gr}_{H,x}\overset{G(O_x)}{\times}\textrm{Bun}_G^{\infty\cdot x}$$ where $\textrm{Bun}_G^{\infty\cdot x}$ parametrizes $G$-bundles on $X$ equipped with a trivialization on the formal disk around $x$. Similarly, we may write $$\textrm{Hecke}_x^G\simeq \textrm{Gr}_{G,x}\overset{G(O_x)}{\times}\textrm{Bun}_G^{\infty\cdot x}.$$

The map of affine grassmannians $\textrm{Gr}_{G,x}\rightarrow \textrm{Gr}_{H,x}$ realizes $\textrm{Gr}_{G,x}$ as a union of connected components of $\textrm{Gr}_{H,x}$. Thus, so is the case for the map $$j: \textrm{Gr}_{G,x}\overset{G(O_x)}{\times}\textrm{Bun}_G^{\infty\cdot x}\rightarrow \textrm{Gr}_{H,x}\overset{G(O_x)}{\times}\textrm{Bun}_G^{\infty\cdot x}.$$

\noindent As a consequence, the sheaf
\begin{equation}\label{eq:smallshf}
(p_r)_!\circ j_!\circ j^!\big((\overleftarrow{h}_H\circ p_l)^!(\mathcal{F})\overset{!}{\otimes} (q_H\circ p_l)^!(\mathcal{H})\big)
\end{equation}
is a direct summand of the sheaf (\ref{eq:bigshf}). However, the sheaf (\ref{eq:smallshf}) may be rewritten as 
\[
f_{\textrm{Sph}}^!(\mathcal{H})\overset{\textrm{Sph}_G}{\star}
f^!(\mathcal{F}).
\]

\noindent This proves the first assertion of the lemma. It is now clear that $f^!$ maps $D(\textrm{Bun}_H)^{\textrm{quasi-H-temp}}$ to $D(\textrm{Bun}_G)^{\textrm{quasi-G-temp}}$. Indeed, this follows from Remark \ref{r:kernel}.
\end{proof}

\subsubsection{}

Let us end this section by proving one of the inclusions in Theorem A. Namely, that maximally anti-tempered D-modules have singular support contained in $(T^*\textrm{Bun}_G)_{\mathfrak{z}(\mathfrak{g})}$.

\begin{thm}\label{t:A1}

For any reductive group $G$, we have an inclusion $$D(\on{Bun}_G)^{\on{max-anti-temp}}\subset D_{\mathfrak{z}(\mathfrak{g})}(\on{Bun}_G).$$
\end{thm}

\begin{proof}
First, suppose $G$ is semisimple, and let $\cF\in D(\on{Bun}_G)^{\on{max-anti-temp}}$. We need to show that $\cF\in D_0(\on{Bun}_G)$. That is, we need to show that for every scheme $S$ mapping smoothly to $\on{Bun}_G$ via some $f: S\to \on{Bun}_G$, we have $f^!(\cF)\in D_0(S)$. Recall that the affine Grassmannian $\textrm{Gr}_{G,x}$ parametrizes pairs $(\mathcal{P},\tau)$ where $\mathcal{P}$ is a $G$-bundle on $X$ and $\tau$ is a trivialization of $\mathcal{P}$ on $X-x$.

We have a canonical map $$\pi_G: \textrm{Gr}_{G,x}\rightarrow \textrm{Bun}_G, \:\: (\mathcal{P},\tau)\mapsto \mathcal{P}.$$ Let $\mathcal{P}_f$ be the $G$-bundle on $X\times S$ comprising $f$. By \cite[Thm 3.]{drinfeld1995b}, there exists an étale cover $h: S'\rightarrow S$ such that the pullback of $\mathcal{P}_f$ to $X\times S'$ is trivializable on $(X-x)\times S'$. In particular, there exists a map $g: S'\rightarrow \textrm{Gr}_{G,x}$ making the diagram commute:
\[\begin{tikzcd}
	{S'} && {\textrm{Gr}_{G,x}} \\
	\\
	S && {\textrm{Bun}_G.}
	\arrow["f"', from=3-1, to=3-3]
	\arrow["{\pi_G}", from=1-3, to=3-3]
	\arrow["h"', from=1-1, to=3-1]
	\arrow["g", from=1-1, to=1-3]
\end{tikzcd}\]

\noindent Since $h$ is an étale cover, it suffices to show that $h^!\circ f^!(\cF)\simeq g^!\circ \pi_G^!(\cF)$ lands in $D_0(S')$. However, $\pi_G^!$ is $\on{Sph}_G$-linear, and so $\pi_G^!(\cF)$ is maximally anti-tempered. Hence it can be written as a colimit of $\omega_{\on{Gr}_{G,x}}$ by Corollary \ref{c:mon}. It follows immediately that $g^!\circ\pi_G^!(\cF)\in D_0(S')$.

Next, suppose $G\simeq T\times H$ is the product of a torus and a semisimple algebraic group $H$. By Proposition \ref{p:funct} (b) (taking $\cC=D(\on{Bun}_T^{\infty\cdot x}), \cD=D(\on{Bun}_H^{\infty\cdot x})$), we have
$$D(\on{Bun}_G)^{\on{max-anti-G-temp}}\simeq D(\on{Bun}_T)\otimes D(\on{Bun}_H)^{\on{max-anti-H-temp}},$$
which is clearly contained in $D_{\fz(\fg)}(\on{Bun}_G)$ by the above.

Finally, let $G$ be an arbitrary reductive group. Consider the central isogeny $Z(G)^{\circ}\times [G,G]\to G$ and the induced smooth map of moduli stacks $\alpha_G: \on{Bun}_{Z(G)^{\circ}\times [G,G]}\to \on{Bun}_G$. Note that $\alpha_G$ surjects onto a union of connected components of $\on{Bun}_G$ (see e.g \cite{hoffmann2010moduli}).
Let $\cF\in D(\on{Bun}_G)^{\on{max-anti-temp}}$. We may assume that $\cF$ is supported on a single connected component of $\on{Bun}_G$. Now, choose a dominant coweight $\check{\lambda}$ of $G$ such that when we act by the corresponding Hecke functor, $V^{\check{\lambda}}$, the sheaf $V^{\check{\lambda}}\star \cF$ is supported on a single connected component in the image of $\alpha_G$.\footnote{This may easily be done: We have $\pi_0(\on{Bun}_G)=\pi_1(G)=X_*(T)/X_{\on{coroots}}(T)$, and Hecke functors $V^{\check{\lambda}}: D(\on{Bun}_G)\to D(\on{Bun}_G)$ permute the connected components compatibly with the abelian group structure on $X_*(T)/X_{\on{coroots}}(T)$.} By Proposition \ref{p:funct} (a), $\alpha_G^!(V^{\check{\lambda}}\star \cF)$ lies in $D_{\fz(\fg)}(\on{Bun}_{Z(G)^{\circ}\times [G,G]})$. It follows that $V^{\check{\lambda}}\star \cF\in D_{\fz(\fg)}(\on{Bun}_G)$.

At last, we note that $\cF$ occurs as a direct summand of $V^{-w_0(\check{\lambda})}\star (V^{\check{\lambda}}\star \cF)\simeq (V^{-w_0(\check{\lambda})}\star V^{\check{\lambda}})\star \cF$. Since Hecke functors preserve singular support (see \S 2.2.6), this sheaf lies in $D_{\fz(\fg)}(\on{Bun}_G)$.

\end{proof}

\section{Proofs} In this section we finish the proof of Theorem A and then deduce Theorem B as a consequence. We will show that the inclusion $$ D_{\mathfrak{z}(\mathfrak{g})}(\textrm{Bun}_G)\subset D(\textrm{Bun}_G)^{\textrm{max-anti-temp}}$$ for simply-connected semisimple groups and then extract the general case using the formalism in $\S 3$ above. Thus we begin this section by studying the category $D_{\mathfrak{z}(\mathfrak{g})}(\textrm{Bun}_G)=D_0(\textrm{Bun}_G)$ when $G$ is semisimple and simply-connected.

\subsection{Local systems on $\textrm{Bun}_G$ when $G$ is simply-connected}

The goal of this subsection is to prove Proposition \ref{p:cst} below stating that any irreducible local system on $\textrm{Bun}_G$ is trivial whenever $G$ is semisimple and simply-connected. We will prove this by reducing to the case when our ground field $k$ is the complex numbers via standard Noetherian approximation arguments.

\subsubsection{Constructibility} 

Recall that for a smooth algebraic stack $\mathcal{Y}$, $D(\mathcal{Y})^{\textrm{coh}}$ denotes the small category of \emph{locally coherent} D-modules on $\mathcal{Y}$. That is, $D(\mathcal{Y})^{\textrm{coh}}$ consists of those $\mathcal{F}\in D(\mathcal{Y})$ such that for any smooth map from a finite type affine scheme $S\rightarrow \mathcal{Y}$, $f^!(\mathcal{F})\in D(S)$ is a bounded complex with coherent cohomologies.

\subsubsection{Base-change under field extensions}
Let $\tilde{k}\subset k$ be a sub field, and let $\mathcal{\tilde{Y}}$ be a smooth algebraic stack over $\tilde{k}$ such that $D(\mathcal{\tilde{Y}})$ is dualizable. Denote by $\mathcal{Y}$ the base-change of $\tilde{\mathcal{Y}}$ to $k$. The canonical functor $$D(\mathcal{\tilde{Y}})\underset{\textrm{Vect}_{\tilde{k}}}{\otimes}\textrm{Vect}_k\rightarrow D(\mathcal{Y})$$ is an equivalence. However, this is no longer true when one imposes singular support conditions. For example, we have a canonical functor $$D_0(\mathcal{\tilde{Y}})\underset{\textrm{Vect}_{\tilde{k}}}{\otimes}\textrm{Vect}_k\rightarrow D_0(\mathcal{Y}).$$ 

\noindent This functor is fully faithful,\footnote{At least whenever $D_0(\mathcal{\tilde{Y}})$ is dualizable, which will always be the case for us.} but need not be essentially surjective. That is, not every $k$-local system on $\mathcal{Y}$ is induced from a $\tilde{k}$-local system on $\mathcal{\tilde{Y}}$.

\subsubsection{} The following approximation lemma will be important for us.

\begin{lem}\label{l:field}
Let $\mathcal{F}\in D_0(\on{Bun}_G)^{\heartsuit,\on{coh}}$. There exists a subfield $\tilde{k}\subset k$ such that:
\begin{itemize}
    \item $\tilde{k}$ embeds into the complex numbers.
    \item There exists a reductive group $\tilde{G}/\tilde{k}$ and a smooth projective curve $\tilde{X}/\tilde{k}$ such that $$\tilde{G}\underset{{\on{Spec}(\tilde{k})}}{\times}\on{Spec}(k)=G, \:\: \tilde{X}\underset{{\on{Spec}(\tilde{k})}}{\times}\on{Spec}(k)=X.$$
    \item There exists a sheaf $\tilde{\mathcal{F}}\in D_0(\on{Bun}_{\tilde{G}}(\tilde{X}))^{\heartsuit,\on{coh}}$ such that the image of $\tilde{\mathcal{F}}$ under the functor $$D_0(\on{Bun}_{\tilde{G}}(\tilde{X}))\rightarrow D_0(\on{Bun}_{\tilde{G}}(\tilde{X}))\underset{\on{Vect}_{\tilde{k}}}{\otimes}\on{Vect}_k\subset D_0(\on{Bun}_G(X))$$
    is $\mathcal{F}$.
\end{itemize}
\end{lem}
\begin{proof}
First, let $T$ be a smooth scheme of finite type over $k$ and let $\mathcal{F}\in D_0(T)^{\heartsuit,c}$.\footnote{That is, $\mathcal{F}$ is a de Rham local system on $T$.} In this case, we may find a sub field $\tilde{k}\subset k$ that is finitely generated over the rational numbers and such that:
\begin{itemize}
    \item There exists a smooth scheme $\tilde{T}/\tilde{k}$ of finite type satisfying $\tilde{T}\underset{\textrm{Spec}(\tilde{k})}{\times}\textrm{Spec}(k)=T$.
    \item There exists $\tilde{\mathcal{F}}\in D_0(\tilde{T})^{\heartsuit,c}$ such that $\tilde{\mathcal{F}}$ maps to $\mathcal{F}$ under the composition 
    $$D_0(\tilde{T})\rightarrow D_0(\tilde{T})\underset{\textrm{Vect}_{\tilde{k}}}{\otimes}\textrm{Vect}_k\subset D_0(T).$$
\end{itemize}

\noindent Indeed, this follows from standard Noetherian approximation of schemes and coherent sheaves with connections, respectively.

Now, let $\tilde{k}\subset k$ be a sub field \emph{countably} generated over $\mathbb{Q}$ (in particular, it embeds into $\mathbb{C}$), and such that there exists a reductive group $\tilde{G}/\tilde{k}$ and a smooth projective curve $\tilde{X}/\tilde{k}$ satisfying $G=\tilde{G}\underset{{\textrm{Spec}(\tilde{k})}}{\times}\textrm{Spec}(k)$, $X=\tilde{X}\underset{{\textrm{Spec}(\tilde{k})}}{\times}\textrm{Spec}(k)$. Let $\tilde{S}=\sqcup_i \tilde{S}_i\rightarrow \textrm{Bun}_{\tilde{G}}(\tilde{X})$ be a countable disjoint union of finite type schemes $\tilde{S}_i/\tilde{k}$ that smoothly covers $\textrm{Bun}_{\tilde{G}}(\tilde{X})$. Consider the Cartesian square
\[\begin{tikzcd}
	{\tilde{S}} && S \\
	\\
	{\textrm{Bun}_{\tilde{G}}(\tilde{X})} && {\textrm{Bun}_G(X)}
	\arrow["{\tilde{\pi}}"', from=1-1, to=3-1]
	\arrow[from=3-3, to=3-1]
	\arrow["\pi", from=1-3, to=3-3]
	\arrow[from=1-3, to=1-1]
\end{tikzcd}\]

\noindent where $S=\tilde{S}\underset{{\textrm{Spec}(\tilde{k})}}{\times}\textrm{Spec}(k)$. Let $\mathcal{F}\in D_0(\textrm{Bun}_G(X))^{\heartsuit, \textrm{coh}}$ and write $\mathcal{G}:=\pi^!(\mathcal{F})\in D_0(S)^{\textrm{coh}}$. By the above (noting that $\pi^!$ is t-exact up to a cohomological shift), we may assume there exists $\tilde{\mathcal{G}}\in D_0(\tilde{S})^{\textrm{coh}}$ such that its image under the functor $$D_0(\tilde{S})\rightarrow D_0(\tilde{S})\underset{{\textrm{Vect}_{\tilde{k}}}}{\otimes}\textrm{Vect}_k\subset D_0(S)$$ is $\mathcal{G}$. We want to show that $\tilde{\mathcal{G}}$ descends to a sheaf on $\textrm{Bun}_{\tilde{G}}(\tilde{X})$. Consider the commutative diagram:

\[\begin{tikzcd}
	\vdots && \vdots \\
	\\
	{D_0(\tilde{S}\underset{{\textrm{Bun}_{\tilde{G}}(\tilde{X})}}{\times}\tilde{S})\underset{{\textrm{Vect}_{\tilde{k}}}}{\otimes}\textrm{Vect}_k} && {D_0(S\underset{{\textrm{Bun}_G(X)}}{\times}S)} \\
	\\
	{D_0(\tilde{S})\underset{{\textrm{Vect}_{\tilde{k}}}}{\otimes}\textrm{Vect}_k} && {D_0(S)} \\
	\\
	{D_0(\textrm{Bun}_{\tilde{G}}(\tilde{X}))\underset{{\textrm{Vect}_{\tilde{k}}}}{\otimes}\textrm{Vect}_k} && {D_0(\textrm{Bun}_G(X))}
	\arrow[from=5-1, to=5-3]
	\arrow[from=7-1, to=5-1]
	\arrow[from=7-1, to=7-3]
	\arrow[from=7-3, to=5-3]
	\arrow[shift left=1, from=5-1, to=3-1]
	\arrow[shift left=1, from=5-3, to=3-3]
	\arrow[from=3-1, to=3-3]
	\arrow[shift right=1, from=5-1, to=3-1]
	\arrow[shift right=1, from=5-3, to=3-3]
	\arrow[shift left=1, from=3-3, to=1-3]
	\arrow[shift right=1, from=3-3, to=1-3]
	\arrow[from=3-3, to=1-3]
	\arrow[shift right=1, from=3-1, to=1-1]
	\arrow[shift left=1, from=3-1, to=1-1]
	\arrow[from=3-1, to=1-1]
\end{tikzcd}\]

The horizontal functors are fully faithful. For $\tilde{\cG}$ to descend to $\on{Bun}_{\tilde{G}}(\tilde{X})$, we need the descent data of $\cG$ to be defined over $\tilde{k}$. This need not be the case a priori. However, for each isomorphism in the descent datum, we may find an extension of $\tilde{k}$ so that the isomorphism is defined over the extension. Iterating this process, we obtain a field $L$ (still countably generated over $\bQ$) such that the entire descent datum is defined over $L$. Replacing $\tilde{k}$ by $L$ if needed, we may assume $\tilde{\cG}$ descends to a sheaf on $\on{Bun}_{\tilde{G}}(\tilde{X})$.
\end{proof}

\subsubsection{} As a consequence of Lemma \ref{l:field}, we get:

\begin{prop}\label{p:cst}

Let $G$ be a semisimple simply-connected algebraic group. Then any element of $D_0(\on{Bun}_G)^{\heartsuit, \on{coh}}$ is given by direct sums of the constant sheaf $\underline{k}_{\on{Bun}_G}$.

\end{prop}

\begin{proof}
Let us first assume that our ground field is the complex numbers. By Lemma \ref{l:reg} below, elements of $D_0(\textrm{Bun}_G)^{\heartsuit}$ have regular singularities. Thus, we are in the setting of the Riemann-Hilbert correspondence. Denote by $\textrm{Bun}_G^{\textrm{top}}$ the underlying topological stack of the analytic stack $\textrm{Bun}_G(\mathbb{C})$.\footnote{See e.g. \cite[\S 20]{noohi2005foundations} for a rigorous account of passing from complex algebraic stacks to their underlying analytic stacks.} Since $G$ is simply-connected, so is $\textrm{Bun}_G^{\textrm{top}}$ (see e.g. \cite[\S 2]{biswas2021fundamental}). Thus, any irreducible topological local system of $\textrm{Bun}_G^{\textrm{top}}$ is trivial and there are no non-trivial extensions. This proves the assertion in the complex case.

Next, let our ground field $k$ be arbitrary, and let $\mathcal{F}\in D_0(\textrm{Bun}_G)^{\heartsuit, \textrm{coh}}$. By Lemma \ref{l:field}, we may assume that our ground field embeds into $\mathbb{C}$. Write $G_{\mathbb{C}}=G\underset{\textrm{Spec}(k)}{\times}\textrm{Spec}(\mathbb{C})$, $X_{\mathbb{C}}=X\underset{\textrm{Spec}(k)}{\times}\textrm{Spec}(\mathbb{C})$. Since the functor $$D_0(\textrm{Bun}_{G}(X))\underset{\textrm{Vect}_k}{\otimes}\textrm{Vect}_{\mathbb{C}}\rightarrow D_0(\textrm{Bun}_{G_{\mathbb{C}}}(X_{\mathbb{C}}))$$ is fully faithful and t-exact,\footnote{Note that the left hand side admits a natural t-structure coming from the t-structure on $D_0(\textrm{Bun}_G(X))$.} we are done by the complex case.
\end{proof}
\subsubsection{} We used the following lemma in the proof above:

\begin{lem}\label{l:reg}

Let $G$ be a semisimple algebraic group. Then any $\mathcal{F}\in D_0(\on{Bun}_G)$ is regular holonomic.\footnote{That is, for each $n\in \mathbb{Z}$, any coherent subsheaf $\mathcal{G}\subset H^n(\mathcal{F})$ is (smooth locally) regular holonomic in the usual sense.}
\end{lem}
\begin{proof}
Let $\mathcal{F}\in D_0(\textrm{Bun}_G)$. We want to show that for every scheme $S$ mapping smoothly to $\textrm{Bun}_G$ via some $f: S\rightarrow \textrm{Bun}_G$, the sheaf $f^!(\mathcal{F})$ is regular holonomic. As we saw in the proof of Theorem \ref{t:A1}, there exists an étale cover $h: S'\to S$ and a map $g: S'\rightarrow \on{Gr}_{G,x}$ making the diagram commute:
\[\begin{tikzcd}
	{S'} && {\textrm{Gr}_{G,x}} \\
	\\
	S && {\textrm{Bun}_G.}
	\arrow["f"', from=3-1, to=3-3]
	\arrow["{\pi_G}", from=1-3, to=3-3]
	\arrow["h"', from=1-1, to=3-1]
	\arrow["g", from=1-1, to=1-3]
\end{tikzcd}\]

\noindent Since $h$ is an étale cover, it suffices to check that $h^!\circ f^!(\mathcal{F})$ is a regular holonomic sheaf on $S'$.\footnote{The fact that regularity can be checked on étale (in fact, smooth) covers follows e.g. from the description of regularity in \cite[Definition 4.1.1.]{kashiwara2016regular}.} For a dominant coweight $\check{\mu}$ of $G$, write $\textrm{Gr}_{G,x}^{\check{\mu}}:=G(O_x)\backslash G(O_x)\check{\mu}(t)G(O_x)\subset G(O_x)\backslash G(K_x)=\textrm{Gr}_{G,x}$ for the $G(O_x)$-orbit of $\check{\mu}(t)$ in the affine grassmannian. There exists some dominant coweight $\check{\lambda}$ such that $S'\rightarrow \textrm{Gr}_{G,x}$ factors through $\overline{\textrm{Gr}_{G,x}^{\check{\lambda}}}=\bigsqcup_{\mu\leq \lambda} \textrm{Gr}_{G,x}^{\check{\mu}}\subset \textrm{Gr}_{G,x}$. Pulling back, this provides a finite stratification on $S'$. Since regularity of holonomic D-modules can be checked on a finite stratification,\footnote{The assertion follows by inducting on the number of strata by choosing an open stratum and using the standard localization triangle of D-modules.} we see that it suffices to show that the pullback of $\mathcal{F}$ under the map $$g_{\check{\mu}}: \textrm{Gr}_{G,x}^{\check{\mu}}\rightarrow \textrm{Bun}_G$$ is regular holonomic for any dominant coweight $\check{\mu}$. Noting that $g_{\check{\mu}}^!(\mathcal{F})$ comes from pullback along $\textrm{Gr}_{G,x}^{\check{\mu}}\hookrightarrow \overline{\textrm{Gr}_{G,x}^{\check{\mu}}}$, this follows by choosing a projective resolution of singularities $Y$ of $\overline{\textrm{Gr}_{G,x}^{\check{\mu}}}$ that induces an isomorphism over $\textrm{Gr}_{G,x}^{\check{\mu}}$. Indeed, the pullback of $\cF$ to any smooth scheme has singular support contained in the zero section. Moreover, for a smooth projective scheme $Y$, any $\mathcal{G}\in D_0(Y)$ is automatically regular holonomic.
\end{proof}

\begin{rem}\label{r:red}

It is not difficult to show that Lemma \ref{l:reg} also holds whenever $G$ is an arbitrary reductive group. Indeed, from the isogeny $Z(G)^{\circ}\times [G,G]\rightarrow G$, we may reduce to the case where $G$ is the product of a torus and a semisimple algebraic group (by the same reasoning as in the proof of Theorem \ref{t:A3} below). Since $\textrm{Bun}_{G_m}\simeq \mathbb{Z}\times \textrm{Jac}(X)\times \mathbb{B}\mathbb{G}_m$ is ind-projective up to a stacky factor, the assertion follows from the same reasoning as above.

We remark that it was shown in \cite{arinkin2020stack} that, in fact, automorphic D-modules with nilpotent singular support are regular holonomic. This result is incredibly deep, however, which is why we do not appeal to it.

\end{rem}

\subsection{Proof of Theorem A (semisimple case)}

Let us proceed with the proof of Theorem A when $G$ is semisimple.

\begin{thm}\label{t:A2}

Let $G$ be a semisimple algebraic group. Then we have an inclusion 
\[
D_0(\on{Bun}_G)\subset D(\on{Bun}_G)^{\on{max-anti-temp}}.
\]

\end{thm}
\begin{proof}

\emph{Step 1.} Assume in this first step that $G$ is semisimple and simply-connected, and let $\mathcal{F}\in D_{\mathfrak{z}(\mathfrak{g})}(\textrm{Bun}_G)=D_0(\textrm{Bun}_G)$. Note that by Corollary \ref{c:hecke}, the constant sheaf $\underline{k}_{\textrm{Bun}_G}$ belongs to $D(\textrm{Bun}_G)^{\textrm{max-anti-temp}}$. By Proposition \ref{p:cst}, we have $$D_0(\textrm{Bun}_G)^{\heartsuit}\subset D(\textrm{Bun}_G)^{\textrm{max-anti-temp}}.$$

Now, let $\mathcal{F}\in D_0(\textrm{Bun}_G)$ have bounded cohomologies. Then we may write $\mathcal{F}$ as a finite colimit of its cohomology sheaves $H^n(\mathcal{F})[-n]$, which implies that $\mathcal{F}\in D(\textrm{Bun}_G)^{\textrm{max-anti-temp}}$.

If $\mathcal{F}\in D_0(\textrm{Bun}_G)$ is cohomologically bounded from below, then we may write $\mathcal{F}\simeq \underset{n}{\textrm{colim}}\: \tau^{\leq n}(\mathcal{F})$ where each $\tau^{\leq n}(\mathcal{F})$ is cohomologically bounded. Thus, $\mathcal{F}$ is maximally anti-tempered in this case as well.

Finally, let $\mathcal{F}\in D_0(\textrm{Bun}_G)$ be arbitrary. We may write $\mathcal{F}\simeq \underset{n}{\textrm{lim}}\: \tau^{\geq -n}(\mathcal{F})$. Then we have that each $\tau^{\geq -n}(\mathcal{F})$ is cohomologically bounded from below and hence maximally anti-tempered. Since the functor $$\textrm{quasi-temp}: D(\textrm{Bun}_G)\rightarrow D(\textrm{Bun}_G)^{\textrm{quasi-temp}}$$
commutes with limits (being a right adjoint), we have $\textrm{quasi-temp}(\mathcal{F})=0$.\\

\emph{Step 2.} Assume now that $G$ is semisimple but not necessarily simply-connected. Let $\Tilde{G}\rightarrow G$ be a universal covering so that $\Tilde{G}$ is semisimple simply-connected, and let $\mathcal{F}\in D_{\mathfrak{z}(\mathfrak{g})}(\textrm{Bun}_G)=D_0(\textrm{Bun}_G)$.

We need to show that $\tilde{F}:=\textrm{quasi-temp}(\mathcal{F})$ is zero. From the central isogeny $\tilde{G}\rightarrow G$, we get an induced smooth map $f: \textrm{Bun}_{\tilde{G}}\rightarrow \textrm{Bun}_G$. Note first that $\tilde{F}\in D_0(\textrm{Bun}_G)$. Indeed, the functor $D(\textrm{Bun}_G)\rightarrow D(\textrm{Bun}_G)^{\textrm{quasi-temp}}$ is given by convolving with a Hecke functor at $x$ (see $\S 2.3.4$), and these preserve singular support, cf. $\S 2.2.6$. By Lemma \ref{l:!}, we have $f^!(\mathcal{\tilde{F}})\in D_0(\textrm{Bun}_{\tilde{G}})$, and so $f^!(\mathcal{\tilde{F}})$ is maximally anti-tempered by Step 1. On the other hand, Lemma \ref{l:key} says that $f^!(\mathcal{\tilde{F}})$ is quasi-tempered. Hence $f^!(\mathcal{\tilde{F}})=0$.

Thus, if $\mathcal{\tilde{F}}$ is not supported away from the image of $f$, this finishes the proof. If $\mathcal{\tilde{F}}$ is supported away from $\textrm{Im(f)}$, choose a dominant coweight $\check{\lambda}$ such that $V^{\check{\lambda}}\star \mathcal{\tilde{F}}$ is not supported away from $\textrm{Im}(f)$. Note also that $V^{\check{\lambda}}\star \mathcal{\tilde{F}}\neq 0$ whenever $\mathcal{\tilde{F}}\neq 0$ because $\mathcal{\tilde{F}}$ occurs as a direct summand of $V^{\check{-w_0(\check{\lambda}})}\star (V^{\check{\lambda}}\star \mathcal{\tilde{F}})$. However, by the above argument $f^!(V^{\check{\lambda}}\star \mathcal{\tilde{F}})=0$, which implies that $\mathcal{\tilde{F}}=0$.
\end{proof}

\subsection{Proof of Theorem A for reductive groups} We are now able to prove Theorem A for a general reductive group $G$. Let us start by reviewing a couple of preliminary lemmas.

\begin{lem}\label{l: tensor}
Let $\mathcal{Y}$ be a smooth algebraic stack and $\mathcal{N}$ a Zariski-closed conical subset of $T^*\mathcal{Y}$. For $\mathcal{F}\in D_0(\mathcal{Y})$ \emph{and} $\mathcal{G}\in D_{\mathcal{N}}(\mathcal{Y})$, the tensor product $\mathcal{F}\overset{!}{\otimes} \mathcal{G}$ lies in $D_{\mathcal{N}}(\mathcal{Y})$.

\end{lem}
\begin{proof}
We may assume that $\mathcal{Y}$ is a smooth scheme. Consider the diagonal map $\Delta: \mathcal{Y}\rightarrow \mathcal{Y}\times \mathcal{Y}$. We need to show that $$\Delta^!(\mathcal{F}\boxtimes \mathcal{G})\in D_{\mathcal{N}}(\mathcal{Y}).$$

\noindent More precisely, we need to show that if $\mathcal{L}$ is a coherent subsheaf of $H^n(\mathcal{F}\overset{!}{\otimes} \mathcal{G})$, then the singular support of $\mathcal{L}$ is contained in $\mathcal{N}$.

First, note that the sheaf $\mathcal{F}\boxtimes \mathcal{G}$ is \emph{non-characteristic} with respect to $\Delta$.\footnote{That is, the codifferential $d\Delta^{\dot}: SS(\mathcal{F}\boxtimes \mathcal{G})\times_{\mathcal{Y}\times\mathcal{Y}}\mathcal Y\rightarrow T^*\mathcal{Y}$ is finite over its image.} It follows that for any coherent subsheaf $\mathcal{H}\subset H^n(\mathcal{F}\boxtimes\mathcal{G})$, $\Delta^!(\mathcal{H})$ is still coherent and concentrated in degree $\textrm{dim}\: \textrm{Bun}_G$. In particular, choosing a good filtration for $\mathcal{H}$, the induced filtration on $\Delta^!(\mathcal{H})$ is also good, and so the singular support of $\Delta^!(\mathcal{H})$ coincides with the image of the codifferential 
\[
d\Delta^{\dot}: SS(\mathcal{H})\underset{\mathcal{Y}\times\mathcal{Y}}{\times}\mathcal Y\rightarrow T^*\mathcal{Y}.
\]

\noindent Since $\mathcal{F}\boxtimes\mathcal{G}\in D_{0\times\mathcal{N}}(\mathcal{Y}\times\mathcal{Y})$, this image is contained in $\mathcal{N}$. In particular, if $\mathcal{L}$ is a coherent subsheaf of $H^n(\mathcal{F}\overset{!}{\otimes}\mathcal{G})\simeq \Delta^!(H^{n-\textrm{dim}\:\mathcal{Y}}(\mathcal{F}\boxtimes\mathcal{G}))$, the singular support of $\mathcal{L}$ is contained in $\mathcal{N}$.
\end{proof}

\subsubsection{} We will also need the following lemma:

\begin{lem}\label{l:top}
Let $\mathcal{Y}$ be an algebraic stack eventually coconnective locally almost of finite type, and let $\mathcal{C}$ be a dualizable DG category. For an object $\mathcal{F}\in \on{QCoh}(\mathcal{Y})\otimes \mathcal{C},$ if for all points $i_y: \on{spec}(k')\rightarrow \mathcal{Y}$ of the underlying topological space of $\mathcal{Y}, \mathcal{F}$ gets mapped to zero under the functor $$\on{QCoh}(\mathcal{Y})\otimes \mathcal{C}\rightarrow \on{Vect}_{k'}\otimes \mathcal{C}$$
then $\mathcal{F}$ is zero.

\end{lem}
\begin{proof}
This is proved in \cite[Prop. 4.2.2.1]{faergeman2021arinkin}.
\end{proof}

\subsubsection{Geometric Langlands for a torus.}

In the proof of Theorem A below, we will start by reducing to the case where our group $G$ can be written as a product $G\simeq T \times H$ of a torus $T$ and a semisimple algebraic group $H$. In this case, we have $D_{\mathfrak{z}(\mathfrak{g})}(\textrm{Bun}_G)\simeq D_{T^*\textrm{Bun}_T\times 0}(\textrm{Bun}_T\times \textrm{Bun}_H)$. We are able to control this category by $\S 4.2$ above and geometric Langlands for $T$.

Namely, recall that the Fourier-Mukai-Laumon transform provides an equivalence $$\mathbb{L}_T: D(\textrm{Bun}_T)\simeq \textrm{QCoh}(\on{LocSys}_{\check{T}})$$
such that for each $\check{T}$-local system $\sigma$ on $X$, we get a $1$-dimensional local system $E_{\sigma}$ on $\textrm{Bun}_T$. Moreover, taking the fiber at $\sigma$ corresponds under $\mathbb{L}_T$ to the functor $$D(\textrm{Bun}_T)\rightarrow \textrm{Vect},\:\: \mathcal{F}\mapsto H_{\textrm{ren}-\textrm{dR}}^*(\textrm{Bun}_T,E_{\sigma}\overset{*}{\otimes} \mathcal{F}).$$

\subsubsection{Field extensions}

For a field extension $k'/k$ and a stack $\mathcal{Y}$, denote by $\mathcal{Y}'=\mathcal{Y}\underset{\textrm{Spec}(k)}{\times}\textrm{Spec}(k')$ the base-change of $\mathcal{Y}$ to $k'$. Similarly, for a $k$-linear DG category $\mathcal{C}$, denote by $\mathcal{C}':=\mathcal{C}\underset{\textrm{Vect}_k}{\otimes}\textrm{Vect}_{k'}.$

For a $\check{T}'$-local system $\sigma'$ on $X'$ (that is, a map $\textrm{Spec}(k')\rightarrow \textrm{LocSys}_{\check{T}'}(X')$), we denote by $E_{\sigma'}$ the corresponding local system on $\textrm{Bun}_{T'}(X')$. From Lemma \ref{l:top}, we get:

\begin{cor}\label{c:top}
Let $\mathcal{C}$ be a dualizable DG category and $\mathcal{F}$ an element of $D(\on{Bun}_T(X))\otimes \mathcal{C}$. If for all geometric points $\on{Spec}(k')\rightarrow \on{LocSys}_{\check{T}}(X)$, the image of $\mathcal{F}$ under the composition $$D(\on{Bun}_T(X))\underset{\on{Vect}_k}{\otimes} \mathcal{C}\rightarrow D(\on{Bun}_{T'}(X'))\underset{\on{Vect}_{k'}}{\otimes} \mathcal{C}'\xrightarrow{H_{\on{ren}-\on{dR}}^*(\on{Bun}_{T'}(X'), E_{\sigma'}\overset{*}{\otimes} -)\otimes \on{Id}}\mathcal{C}'$$
is zero, then $\mathcal{F}$ is also zero.

\end{cor}

\subsubsection{} 

Finally, let us note that quasi-temperedness plays well under base change (see also \cite[\S 10.3]{faergeman2022non}). Namely, for a field extension $k'/k$, we have an identification $$\textrm{Sph}_G\underset{\textrm{Vect}_k}{\otimes}\textrm{Vect}_{k'}\simeq \textrm{Sph}_{G'}$$ as $k'$-linear DG categories.

For each DG category $\mathcal{C}$ equipped with an action of $\textrm{Sph}_G$, it follows from the definition (\ref{eq:qtemp}) that we get an equivalence\footnote{Note that there is nothing special about quasi-temperedness here. We have analogues equivalences when considering e.g. maximal anti-temperedness or temperedness.} $$\mathcal{C}^{\textrm{quasi-G-temp}}\underset{\textrm{Vect}_k}{\otimes}\textrm{Vect}_{k'}\simeq (\mathcal{C}')^{\textrm{quasi-G}'\textrm{-temp}}.$$

\subsubsection{} Let us finish the proof of Theorem A.

\begin{thm}\label{t:A3}
We have an inclusion $$D_{\mathfrak{z}(\mathfrak{g})}(\on{Bun}_G)\subset D(\on{Bun}_G)^{\on{max-anti-temp}}.$$
\end{thm}
\begin{proof}

\emph{Step 1.} Let us prove the assertion when $G\simeq T\times H$ is the product of a torus and a semisimple algebraic group. We will reduce to this case in Step 2 below.

In this case, we have $$D_{\mathfrak{z}(\mathfrak{g})}(\textrm{Bun}_G)\simeq D_{T^*\textrm{Bun}_T\times 0}(\textrm{Bun}_T\times \textrm{Bun}_H).$$ 

\noindent Thus, let $\mathcal{F}\in D_{T^*\textrm{Bun}_T\times 0}(\textrm{Bun}_T\times \textrm{Bun}_H)$ and write $\tilde{\mathcal{F}}=\textrm{quasi-temp}(\mathcal{F})$. We need to show that $\tilde{\mathcal{F}}=0$. By Corollary \ref{c:top}, it suffices to show that for each $k'$-linear $\check{T}'$-local system on $X'$, $\sigma': \textrm{Spec}(k')\rightarrow \textrm{LocSys}_{\check{T}}$, the composition $$D(\textrm{Bun}_T)\underset{\textrm{Vect}_k}{\otimes} D(\textrm{Bun}_H)\rightarrow$$ $$\rightarrow D(\textrm{Bun}_{T'}(X'))\underset{\textrm{Vect}_{k'}}{\otimes}D(\textrm{Bun}_{H'}(X')) \xrightarrow{H_{\textrm{ren}-\textrm{dR}}^*(\textrm{Bun}_{T'}(X'), E_{\sigma'}\overset{*}{\otimes} -)\otimes \textrm{Id}} D(\textrm{Bun}_{H'}(X'))$$ kills $\tilde{\mathcal{F}}$.

First, observe that the sheaf $\tilde{\mathcal{F}}$ still lies in $D_{T^*\textrm{Bun}_T\times 0}(\textrm{Bun}_T\times \textrm{Bun}_H)$. As previously noted, this follows from the fact that quasi-temp is given by a Hecke functor at $x$. Next, note that the base-change functor
\begin{equation}\label{eq:bc}
D(\textrm{Bun}_T\times \textrm{Bun}_H)\rightarrow D(\textrm{Bun}_{T'}(X')\times \textrm{Bun}_{H'}(X'))
\end{equation}
restricts to a functor 
\begin{equation}
D_{T^*\textrm{Bun}_T\times 0}(\textrm{Bun}_T\times\textrm{Bun}_H)\rightarrow D_{T^*\textrm{Bun}_{T'}(X')\times 0}(\textrm{Bun}_{T'}(X')\times\textrm{Bun}_{H'}(X')).
\end{equation}

The functor (\ref{eq:bc}) sends $\tilde{\mathcal{F}}$ to a $(T'\times H')$-quasi-tempered object by $\S 4.3.4$ above. Combining these assertions, we see that in the composition above, we only need to consider $\check{T}$-local systems $\sigma$ defined over $k$.

Thus, let $\sigma: \textrm{Spec}(k)\rightarrow \textrm{LocSys}_{\check{T}}$ be a $\check{T}$-local system on $X$. Denote by $p$ the projection $\textrm{Bun}_T\times \textrm{Bun}_H\rightarrow \textrm{Bun}_H$. We need to show that $$p_{*,\textrm{ren}-\textrm{dR}}((E_{\sigma}\boxtimes \underline{k}_{\textrm{Bun}_H})\overset{*}{\otimes}\tilde{\mathcal{F}})=0.$$

\noindent By Lemma \ref{l: tensor} above and Lemma \ref{l:ss} below, we see that this sheaf has singular support contained in the zero section of $\textrm{Bun}_H$. That is, it is maximally anti-tempered by Theorem A for the group $H$. On the other hand, we noted in Remark \ref{r:product} that 
$$D(\textrm{Bun}_T\times \textrm{Bun}_H)^{\textrm{quasi-T}\times \textrm{H-temp}}\simeq D(\textrm{Bun}_T)\otimes D(\textrm{Bun}_H)^{\textrm{quasi-H-temp}}.$$

\noindent Thus, $p_{*,\textrm{ren}-\textrm{dR}}((E_{\sigma}\boxtimes k_{\textrm{Bun}_H})\overset{*}{\otimes}\tilde{\mathcal{F}})$ is quasi-tempered as well. Hence it is zero.\\

\emph{Step 2.} Let us conclude Theorem A from Step 1. Let $G$ be an arbitrary reductive group. We mimic the proof of Theorem \ref{t:A2}, step 2. Namely, consider the central isogeny $Z(G)^{\circ}\times [G,G]\rightarrow G$ inducing a smooth map $$\alpha_G: \textrm{Bun}_{Z(G)^{\circ}\times [G,G]}\rightarrow \textrm{Bun}_G$$ realizing the image as a union of connected components $\textrm{Bun}_G$.

Choose a dominant coweight $\check{\lambda}$ such that $V^{\check{\lambda}}\star \mathcal{\tilde{F}}$ is not supported away from $\textrm{Im}(\alpha_G)$. By Lemma \ref{l:key}, $\alpha_G^!(V^{\check{\lambda}}\star \mathcal{\tilde{F}})$ is quasi-tempered.

On the other hand, $\alpha_G^!(V^{\check{\lambda}}\star \mathcal{\tilde{F}})$ lies in $D_{T^*\textrm{Bun}_{Z(G)^{\circ}}\times 0}(\textrm{Bun}_{Z(G)^{\circ}}\times \textrm{Bun}_{[G,G]})$ and so is maximally anti-tempered by the above. Thus $\alpha_G^!(V^{\check{\lambda}}\star \mathcal{\tilde{F}})=0$, and since $V^{\check{\lambda}}\star \mathcal{\tilde{F}}$ is not supported away from $\textrm{Im}(\alpha_G)$, this forces $V^{\check{\lambda}}\star \mathcal{\tilde{F}}=0$. This, in turn, implies that $\mathcal{\tilde{F}}=0$.
\end{proof}

\subsubsection{} We used the following lemma in Step 1 above.

\begin{lem}\label{l:ss}

Let $\mathcal{Y}$ be a smooth algebraic stack and denote by $p$ the projection $\on{Bun}_T\times \mathcal{Y}\rightarrow \mathcal{Y}$. Then for any Zariski-closed conical $\mathcal{N}\subset T^*\mathcal{Y}$, the functor $p_{*,\on{ren}-\on{dR}}: D(\on{Bun}_T\times \mathcal{Y})\rightarrow D(\mathcal{Y})$ restricts to a functor $$p_{*,\on{ren}-\on{dR}}:D_{T^*\on{Bun}_T\times \mathcal{N}}(\on{Bun}_T\times \mathcal{Y})\rightarrow D_{\mathcal{N}}(\mathcal{Y}).$$
\end{lem}

\begin{proof}
For notational purposes, we will assume $T=\bG_m$; this is clearly harmless.

Write $\textrm{Bun}_{\mathbb{G}_m}\simeq \mathbb{Z}\times \mathbb{B}\mathbb{G}_m\times \textrm{Jac}$, where $\textrm{Jac}$ is the Jacobian variety of $X$. Note that the functor
$$p_!: D(\on{Bun}_{\bG_m}\times \cY)\simeq D(\on{Bun}_{\bG_m})\otimes D(\cY)\to D(\cY)$$ 
is well-defined. Moreover, since $p_!$ and $p_{*,\on{ren-dR}}$ differ by a cohomological shift (\cite[\S 9.1.6]{drinfeld2013some}), we see that it suffices to show the lemma for the functor $p_!$.

From the smooth cover $\on{pt}\to \bB\bG_m$, we get an induced cover
$$q: \on{Jac}\times \cY\to \bB\bG_m\times \on{Jac}\times \cY.$$
Let us start by showing that $q_!: D(\on{Jac}\times \cY)\to D(\bB\bG_m\times \on{Jac}\times \cY)$ restricts to a functor
\begin{equation}\label{eq:qq}
D_{T^*\on{Jac}\times \cN}(\on{Jac}\times \cY)\to D_{T^*(\bB\bG_m\times \on{Jac})\times \cN}(\bB\bG_m\times \on{Jac}\times \cY).
\end{equation}

\noindent By Lemma \ref{l:!}, it suffices to show that $q^!q_!$ takes $D_{T^*\on{Jac}\times \cN}(\on{Jac}\times \cY)$ to itself. However, by smooth base change, $q^!q_!$ simply tensors by the vector space $H_*(\bG_m)$, and so (\ref{eq:qq}) is clear.

Next, note that the functor (\ref{eq:qq}) generates the target under colimits, since its right adjoint is conservative. Denoting by $\pi_{\bB\bG_m}$ the projection $\bB\bG_m\times \on{Jac}\times \cY\to \on{Jac}\times \cY$, we see that ${\pi_{\bB\bG_m}}_!$ restricts to a functor
$${\pi_{\bB\bG_m}}_!: D_{T^*(\bB\bG_m\times \on{Jac})\times \cN}(\bB\bG_m\times \on{Jac}\times \cY)\to D_{T^*\on{Jac}\times \cN}(\on{Jac}\times \cY).$$

Finally, writing
$$D_{T^*\on{Bun}_{\bG_m}\times \cN}(\on{Bun}_{\bG_m}\times \cY)\simeq D(\bZ)\otimes D_{T^*(\bB\bG_m\times \on{Jac})\times \cN}(\bB\bG_m\times \on{Jac}\times \cY),$$

\noindent we see that the proof follows from Lemma \ref{l:*}, using that $\on{Jac}$ is projective.

\end{proof}

\begin{rem}\label{r:act}
In fact the above proof shows that for any algebraic stack $\cY$, the fully faithful functor $$D(\bB\bG_m)\otimes D_{\cN}(\cY)\to D_{T^*\bB\bG_m\times \cN}(\bB\bG_m\times \cY)$$ is an equivalence. Indeed the proof shows that taking $*$-pushforward along the projection 
$\bB\bG_m\times \cY\to \cY$
restricts to a functor 
\begin{equation}\label{eq:rest}
D_{T^*\bB\bG_m\times \cN}(\bB\bG_m\times \cY)\to D_{\cN}(\cY).
\end{equation}

\noindent Since $D(\bB\bG_m)$ is generated by the dualizing sheaf under colimits,\footnote{This follows from writing $D(\bB\bG_m)\simeq H_*(G)\on{-mod}$ by the Barr-Beck-Lurie theorem, where $\omega_{\bB\bG_m}$ corresponds to the augmentation module $k$ (see e.g \cite[\S 2.4.4]{beraldo2021tempered}).} we see that the functor (\ref{eq:rest}) is conservative. Hence its left adjoint generates $D_{T^*\bB\bG_m\times \cN}(\bB\bG_m\times \cY)$ under colimits. However, this left adjoint clearly factors through $D(\bB\bG_m)\otimes D_{\cN}(\cY)$.
\end{rem}

\subsection{Proof of Theorem B}

In this subsection we prove Theorem B. That is, we prove that the functors
\begin{equation}\label{eq:z}
D_{\mathfrak{z}(\mathfrak{g})}(\textrm{Bun}_G)\otimes D_{\mathfrak{z}(\mathfrak{h})}(\textrm{Bun}_H)\xrightarrow{-\boxtimes -} D_{\mathfrak{z}(\mathfrak{g})\times \mathfrak{z}(\mathfrak{h})}(\textrm{Bun}_G\times \textrm{Bun}_H)
\end{equation}
\begin{equation}\label{eq:0}
D_0(\textrm{Bun}_G)\otimes D_0(\textrm{Bun}_H)\xrightarrow{-\boxtimes -} D_0(\textrm{Bun}_G\times \textrm{Bun}_H)
\end{equation}
are equivalences. Standard arguments show that if $\mathcal{C}_0\hookrightarrow \mathcal{C}$ is a fully faithful embedding of DG categories and $\mathcal{D}$ is a dualizable DG category, then then canonical functor $$\mathcal{C}_0\otimes \mathcal{D}\rightarrow \mathcal{C}\otimes \mathcal{D}$$
is also fully faithful. In particular, using that the functor $$D(\textrm{Bun}_G)\otimes D(\textrm{Bun}_H)\xrightarrow{-\boxtimes -}D(\textrm{Bun}_G\times \textrm{Bun}_H)$$ is an equivalence, we see that if $D_{\mathfrak{z}(\mathfrak{g})}(\textrm{Bun}_G)$ (resp.\:$D_0(\textrm{Bun}_G)$) is dualizable as a DG category, then the functor (\ref{eq:z}) (resp. (\ref{eq:0})) is fully faithful. However, the dualizability of $D_{\mathfrak{z}(\mathfrak{g})}(\textrm{Bun}_G)$ (resp. $D_0(\textrm{Bun}_G)$) follows from Theorem A (resp. Lemma \ref{l:dual} below).

Note that essential surjectivity of (\ref{eq:z}) follows immediately from Theorem A and Proposition \ref{p:funct} (b) (again, taking $\cC=D(\on{Bun}_H^{\infty\cdot x}), \cD=D(\on{Bun}_G^{\infty\cdot x})$). Thus, it remains to prove that the functor (\ref{eq:0}) is essentially surjective. To do this, we will need the following theorem proved in \cite{arinkin2020stack}:

\begin{thm}\label{t:ab}

Let $A$ be an abelian variety, and let $\mathcal{Y}$ be a smooth algebraic stack. Then the functor $$D_0(A)\otimes D_0(\mathcal{Y})\rightarrow D_0(A\times \mathcal{Y})$$ is an equivalence.

\end{thm}
\begin{proof}
This follows from combining Theorem F. 9.7, Proposition E.4.4 and $\S$23.3.2 in \cite{arinkin2020stack}.
\end{proof}

\begin{rem}\label{r:da}
Theorem \ref{t:ab} is in fact a very special case of the results cited above. The assertion of the theorem holds more generally whenever $A$ is a smooth and proper variety and $D_0(A)$ is "duality-adapted" (see \cite[\S E.4]{arinkin2020stack}).
\end{rem}

\subsubsection{} 

\emph{Proof of Theorem B.} Let $G$ and $H$ be two reductive groups. We need to prove that (\ref{eq:0}) is essentially surjective. We proceed in four steps. Namely, we will: 
\begin{enumerate}[label=\alph*]
\item[(a)] Consider the case when $G$ and $H$ are both semisimple.
\item[(b)] Consider the case when $G$ and $H$ are both given by the product of a torus and a semisimple group.
\item[(c)] Consider the case when $G$ is arbitrary and $H$ is semisimple.
\item[(d)] Consider the general case.
\end{enumerate}

\emph{proof of (a):} Assume $G$ and $H$ are both semisimple. Then $\mathfrak{z}(\mathfrak{g})$ and $\mathfrak{z}(\mathfrak{h})$ are trivial, and so (\ref{eq:0}) is equivalent to (\ref{eq:z}).

\emph{proof of (b):} Assume now that we may write $G\simeq T_G\times [G,G]$, $H\simeq T_H\times [H,H]$ where $T_G$ and $T_H$ are tori. Recall that for $T$ a torus, we have $$\textrm{Bun}_T=\big(\mathbb{Z}\times \mathbb{B}\mathbb{G}_m\times \textrm{Jac}\big)^{\times \textrm{rk}(T)}.$$

\noindent Thus, for $\mathcal{Y}$ a smooth algebraic stack, it follows from Theorem \ref{t:ab} that we may write $$D_0(\textrm{Bun}_T\times \mathcal{Y})\simeq D_0(\textrm{Jac}^{\textrm{rk}(T)})\otimes D_0(\mathbb{Z}^{\textrm{rk}(T)}\times \mathbb{B}\mathbb{G}_m^{\textrm{rk}(T)}\times \mathcal{Y}).$$

\noindent Moreover, by Remark \ref{r:act}, we may further rewrite this as
$$D_0(\textrm{Jac}^{\textrm{rk}(T)})\otimes D(\mathbb{Z}^{ \textrm{rk}(T)}\times \mathbb{B}\mathbb{G}_m^{\textrm{rk}(T)})\otimes D_0(\mathcal{Y})\simeq D_0(\textrm{Bun}_T)\otimes D_0(\mathcal{Y}).$$

\noindent This observation combined with (a) yields: $$D_0(\textrm{Bun}_G\times \textrm{Bun}_H)\simeq D_0(\textrm{Bun}_{T_G}\times \textrm{Bun}_{[G,G]}\times \textrm{Bun}_{T_H}\times \textrm{Bun}_{[H,H]})$$ $$\simeq D_0(\textrm{Bun}_{T_G}\times \textrm{Bun}_{[G,G]})\otimes D_0(\textrm{Bun}_{T_H}\times \textrm{Bun}_{[H,H]})\simeq D_0(\textrm{Bun}_G)\otimes D_0(\textrm{Bun}_H).$$

\emph{proof of (c):} Now let $G$ be an arbitrary reductive group and let $H$ be semisimple. Note that from (\ref{eq:z}) above, we have
\begin{equation}\label{eq:inc1}
D_0(\textrm{Bun}_G\times \textrm{Bun}_H)\subset D_{\mathfrak{z}(\mathfrak{g})\times 0}(\textrm{Bun}_G\times \textrm{Bun}_H)\simeq D_{\mathfrak{z}(\mathfrak{g})}(\textrm{Bun}_G)\otimes D_0(\textrm{Bun}_H).
\end{equation}

\noindent Observe that $D_0(\textrm{Bun}_H)=D(\textrm{Bun}_H)^{\textrm{max-anti-temp}}$ is compactly generated and hence dualizable. We want to show that the inclusion (\ref{eq:inc1}) factors through $$D_0(\textrm{Bun}_G)\otimes D_0(\textrm{Bun}_H)\hookrightarrow D_{\mathfrak{z}(\mathfrak{g})}(\textrm{Bun}_G)\otimes D_0(\textrm{Bun}_H).\footnote{Note that this functor is indeed fully faithful since $D_0(\textrm{Bun}_H)$ is dualizable by Theorem A.}$$ 

\noindent It suffices to show that for each $\lambda\in D_0(\textrm{Bun}_H)^{\vee}$, the composition
\begin{equation}\label{eq:comp1}
D_0(\textrm{Bun}_G\times \textrm{Bun}_H)\hookrightarrow D_{\mathfrak{z}(\mathfrak{g})}(\textrm{Bun}_G)\otimes D_0(\textrm{Bun}_H)\xrightarrow{\textrm{Id}\otimes \lambda}D_{\mathfrak{z}(\mathfrak{g})}(\textrm{Bun}_G)
\end{equation}
lands in $D_0(\textrm{Bun}_G)$.

Consider the central isogeny $Z(G)^{\circ}\times [G,G]\rightarrow G$ and the induced map of moduli stacks $\alpha_G: \textrm{Bun}_{Z(G)^{\circ}\times [G,G]}\rightarrow \textrm{Bun}_G$. This map smoothly surjects onto a union of connected components of $\on{Bun}_G$.

We have a commutative diagram
\[\begin{tikzcd}
	{D_0(\textrm{Bun}_G\times\textrm{Bun}_H)} && {D_0(\textrm{Bun}_{Z(G)^{\circ}\times [G,G]})\otimes D_0(\textrm{Bun}_H)} \\
	\\
	{D_{\mathfrak{z}(\mathfrak{g})}(\textrm{Bun}_G)\otimes D_0(\textrm{Bun}_H)} && {D(\textrm{Bun}_{Z(G)^{\circ}\times [G,G]})\otimes D_0(\textrm{Bun}_H)} \\
	\\
	{D_{\mathfrak{z}(\mathfrak{g})}(\textrm{Bun}_G)} && {D(\textrm{Bun}_{Z(G)^{\circ}\times [G,G]}).}
	\arrow["{(\alpha_G\times\textrm{Id})^!}", from=1-1, to=1-3]
	\arrow[hook, from=1-3, to=3-3]
	\arrow[hook, from=1-1, to=3-1]
	\arrow["{\alpha_G^!\otimes\textrm{Id}}", from=3-1, to=3-3]
	\arrow["{\textrm{Id}\otimes\lambda}"', from=3-1, to=5-1]
	\arrow["{\alpha_G^!}", from=5-1, to=5-3]
	\arrow["{\textrm{Id}\otimes\lambda}", from=3-3, to=5-3]
\end{tikzcd}\]

\noindent We have used (b) to rewrite the upper right corner. The diagram shows that the composition $$D_0(\textrm{Bun}_G\times \textrm{Bun}_H)\hookrightarrow D_{\mathfrak{z}(\mathfrak{g})}(\textrm{Bun}_G)\otimes D_0(\textrm{Bun}_H)\xrightarrow{\textrm{Id}\otimes \lambda}D_{\mathfrak{z}(\mathfrak{g})}(\textrm{Bun}_G) \xrightarrow{\alpha_G^!} D(\textrm{Bun}_{Z(G)^{\circ}\times [G,G]})$$ lands in $D_0(\textrm{Bun}_{Z(G)^{\circ}\times [G,G]})$.

Now let $\mathcal{F}\in D_0(\textrm{Bun}_G\times \textrm{Bun}_H)\subset D_{\mathfrak{z}(\mathfrak{g})}(\textrm{Bun}_G)\otimes D(\textrm{Bun}_H)$, and assume that $\cF$ is supported on the connected components coming from the image of $(\alpha_G\times \on{Id}): \on{Bun}_{Z(G)^{\circ}\times [G,G]}\times \on{Bun}_H\to \on{Bun}_G\times \on{Bun}_H$. Then $(\on{Id}\otimes \lambda)(\cF)$ is supported on $\on{Im}(\alpha_G)$, and since $\alpha_G^!((\on{Id}\otimes \lambda)(\cF)$ is an object of $D_0(\on{Bun}_{Z(G)^{\circ}\times [G,G]})$, we see that $(\on{id}\otimes\lambda)(\cF)\in D_0(\on{Bun}_G)$, by Lemma \ref{l:!}. This shows that $\cF\in D_0(\on{Bun}_G)\otimes D(\on{Bun}_H)$, as required.

Next, let $\cF\in D_0(\on{Bun}_G\times \on{Bun}_H)$ be arbitrary and assume for simplicity that $\cF$ is supported on a single connected component. As in Theorem \ref{t:A1}, choose a dominant coweight $\check{\mu}$ of $G\times H$ such that $V^{\check{\mu}}\star \cF$ is supported on $\on{Im}(\alpha_G\times \on{id})$. We have just seen that this implies that $V^{\check{\mu}}\star \cF$ lies in $D_0(\on{Bun}_G)\otimes D_0(\on{Bun}_H)$. Since $\cF$ is a direct summand of $(V^{-w_0(\check{\lambda})}\star V^{\check{\mu}})\star \cF$, so does $\cF$.

\emph{proof of (d):} Finally, let $G$ and $H$ be arbitrary connected reductive groups. We start by showing that for any $\cF\in D_0(\on{Bun}_G\times \on{Bun}_H)$ supported on the connected components in the image of 
$$\alpha_G\times \alpha_H: \on{Bun}_{Z(G)^{\circ}\times [G,G]}\times \on{Bun}_{Z(H)^{\circ}\times [H,H]}\to \on{Bun}_G\times \on{Bun}_H,$$
we have $\cF\in D_0(\on{Bun}_G)\otimes D_0(\on{Bun}_H)$.

Fix such an $\cF$. Let us begin by showing that the image of $\cF$ under inclusion $$D_0(\textrm{Bun}_G\times \textrm{Bun}_H)\hookrightarrow D_{\mathfrak{z}(\mathfrak{g})\times \mathfrak{z}(\mathfrak{h})}(\textrm{Bun}_G\times \textrm{Bun}_H)\simeq  D_{\mathfrak{z}(\mathfrak{g})}(\textrm{Bun}_G)\otimes D_{\mathfrak{z}(\mathfrak{h})}(\textrm{Bun}_H)$$

\noindent lies in in $D_0(\textrm{Bun}_G)\otimes D_{\mathfrak{z}(\mathfrak{h})}(\textrm{Bun}_H)$. It suffices to show that the image of $\cF$ under the composition 
\begin{equation}\label{eq:comp2}
D_0(\textrm{Bun}_G\times \textrm{Bun}_H) \hookrightarrow D_{\mathfrak{z}(\mathfrak{g})}(\textrm{Bun}_G)\otimes D_{\mathfrak{z}(\mathfrak{h})}(\textrm{Bun}_H)\xrightarrow{\textrm{Id}\otimes \lambda} D_{\mathfrak{z}(\mathfrak{g})}(\textrm{Bun}_G)
\end{equation}

\noindent lands in $D_0(\textrm{Bun}_G)$ for all $\lambda\in D_{\mathfrak{z}(\mathfrak{h})}(\textrm{Bun}_H)^{\vee}$. Just as before, we have a diagram:
\[\begin{tikzcd}
	{D_0(\textrm{Bun}_G\times\textrm{Bun}_H)} && {D_0(\textrm{Bun}_{Z(G)^{\circ}\times [G,G]})\otimes D_0(\textrm{Bun}_H)} \\
	\\
	{D_{\mathfrak{z}(\mathfrak{g})}(\textrm{Bun}_G)\otimes D_{\mathfrak{z}(\mathfrak{h})}(\textrm{Bun}_H)} && {D(\textrm{Bun}_{Z(G)^{\circ}\times [G,G]})\otimes D_{\mathfrak{z}(\mathfrak{h})}(\textrm{Bun}_H)} \\
	\\
	{D_{\mathfrak{z}(\mathfrak{g})}(\textrm{Bun}_G)} && {D(\textrm{Bun}_{Z(G)^{\circ}\times [G,G]})}
	\arrow["{(\alpha_G\times\textrm{Id})^!}", from=1-1, to=1-3]
	\arrow[hook, from=1-3, to=3-3]
	\arrow[hook, from=1-1, to=3-1]
	\arrow["{\alpha_G^!\otimes\textrm{Id}}", from=3-1, to=3-3]
	\arrow["{\textrm{Id}\otimes\lambda}"', from=3-1, to=5-1]
	\arrow["{\alpha_G^!}", from=5-1, to=5-3]
	\arrow["{\textrm{Id}\otimes\lambda}", from=3-3, to=5-3]
\end{tikzcd}\]

\noindent Here, we have used that the functor $$D_0(\textrm{Bun}_{Z(G)^{\circ}\times [G,G]})\otimes D_0(\textrm{Bun}_H)\rightarrow D_0(\textrm{Bun}_{Z(G)^{\circ}\times [G,G]}\times \textrm{Bun}_H)$$
is an equivalence to rewrite the upper right corner. Namely, this is a consequence of (c) and the argument in (b) showing that $D_0(\textrm{Bun}_{Z(G)^{\circ}})\otimes D_0(\textrm{Bun}_{[G,G]})\simeq D_0(\textrm{Bun}_{Z(G)^{\circ}\times [G,G]})$.

As in (c), we see that $\alpha_G^!((\on{Id}\otimes \lambda)(\cF))\in D_0(\on{Bun}_{Z(G)^{\circ}\times [G,G]})$, and hence $(\on{Id}\otimes \lambda)(\cF)\in D_0(\on{Bun_G})$. This shows that $\cF\in D_0(\on{Bun}_G)\otimes D_{\fz(\fh)}(\on{Bun}_H)$.

To show that $\cF\in D_0(\on{Bun}_G)\otimes D_0(\on{Bun}_H)$, we need to show that for every $\lambda\in D_0(\on{Bun}_G)^{\vee}$ (note that $D_0(\on{Bun}_G)$ is dualizable by Lemma \ref{l:dual} below), the image of $\cF$ under the map 
$$D_0(\textrm{Bun}_G)\otimes D_{\mathfrak{z}(\mathfrak{h})}(\textrm{Bun}_H)\xrightarrow{\lambda\otimes\textrm{Id}} D_{\mathfrak{z}(\mathfrak{h})}(\textrm{Bun}_H)$$

\noindent lands in $D_0(\on{Bun}_H)$. As above, this follows from the commutative diagram
\[\begin{tikzcd}
	{D_0(\textrm{Bun}_G\times\textrm{Bun}_H)} && {D_0(\textrm{Bun}_G)\otimes D_0(\textrm{Bun}_{Z(H)^{\circ}\times [H,H]})} \\
	\\
	{D_0(\textrm{Bun}_G)\otimes D_{\mathfrak{z}(\mathfrak{h})}(\textrm{Bun}_H)} && {D_0(\textrm{Bun}_G)\otimes D(\textrm{Bun}_{Z(H)^{\circ}\times [H,H]})} \\
	\\
	{D_{\mathfrak{z}(\mathfrak{h})}(\textrm{Bun}_H)} && {D(\textrm{Bun}_{Z(H)^{\circ}\times [H,H]})}
	\arrow["{(\textrm{Id}\times \alpha_H)^!}", from=1-1, to=1-3]
	\arrow[hook, from=1-3, to=3-3]
	\arrow[hook, from=1-1, to=3-1]
	\arrow["{\textrm{Id}\otimes \alpha_H^!}", from=3-1, to=3-3]
	\arrow["{\lambda\otimes\textrm{Id}}"', from=3-1, to=5-1]
	\arrow["{\alpha_H^!}", from=5-1, to=5-3]
	\arrow["{\lambda\otimes\textrm{Id}}", from=3-3, to=5-3]
\end{tikzcd}\]

Now, let $\cF\in D_0(\on{Bun}_G\times \on{Bun}_H)$ be arbitrary (i.e. we no longer require that $\cF$ be supported on $\on{Im}(\alpha_G\times \alpha_H)$). We may suppose $\cF$ is supported on a single connected component. As before, we may choose a dominant coweight $\check{\mu}$ of $G\times H$ such that $V^{\check{\mu}}\star\cF$ is supported on $\on{Im}(\alpha_G\times \on{Id})$. By the above, this implies that $V^{\check{\mu}}\star \cF$ (and hence $\cF$) lies in $D_0(\on{Bun}_G)\otimes D_0(\on{Bun}_H)$.

\qed
\subsubsection{}

We have used the following lemma in the proof of Theorem B:
\begin{lem}\label{l:dual}

The category $D_0(\on{Bun}_G)\hookrightarrow D(\on{Bun}_G)$ is compactly generated. In particular it is dualizable.
\end{lem}

\begin{proof}
If $G$ is semisimple, this assertion follows from Theorem A. Indeed in this case, the inclusion $D_0(\on{Bun}_G)\into D(\on{Bun}_G)$ admits a left adjoint and the category $D(\on{Bun}_G)$ is compactly generated (see \cite{drinfeld2013compact}).

If $G\simeq T\times H$ is the product of a torus and a semisimple group $H$, then as we saw in the proof of Theorem B, step (b), we have $$D_0(\textrm{Bun}_T\times \textrm{Bun}_H)\simeq D_0(\textrm{Bun}_T)\otimes D_0(\textrm{Bun}_H).$$

\noindent We may write $$D_0(\textrm{Bun}_T)\simeq \big(D(\mathbb{Z})\otimes D(\mathbb{B}\mathbb{G}_m)\big)^{\otimes \textrm{rk}(T)}\otimes D_0(\textrm{Jac})^{\otimes \textrm{rk}(T)}.$$
Here, $D_0(\on{Jac})$ is compactly generated because $\textrm{Jac}$ is an abelian variety (see \cite[\S 23]{arinkin2020stack}). This settles the lemma whenever $G$ is the product of a torus and a semisimple group.

Finally, if $G$ is an arbitrary reductive group, consider the central isogeny $Z(G)^{\circ}\times [G,G]\rightarrow G$ and the induced smooth map of stacks $\alpha_G: \textrm{Bun}_{Z(G)^{\circ}\times [G,G]}\rightarrow \textrm{Bun}_G$. Note that the functor 
$$\alpha_{G!}: D_0(\on{Bun}_{Z(G)^{\circ}\times [G,G]})\to D(\on{Bun}_G)$$
lands in $D_0(\on{Bun}_G)$. Indeed, $\alpha_G$ is a $\on{Bun}_{Z([G,G])}$-torsor over its image (see \cite[\S 3]{hoffmann2010moduli}). It follows that the functor $\alpha_G^!\circ \alpha_{G!}$ just tensors with the vector space $H_*(\on{Bun}_{Z([G,G])})$. Thus $\alpha_{G!}(\cF)\in D_0(\on{Bun}_G)$ for any $\cF\in D_0(\on{Bun}_{Z(G)^{\circ}\times [G,G]})$.

It is now easy to see that the collection $\lbrace V^{\check{\lambda}}\star \alpha_{G!}(\cF)\rbrace$ compactly generates $D_0(\on{Bun}_G)$, where $\lbrace \cF\rbrace$ is a collection of compact generators of $D_0(\on{Bun}_{Z(G)^{\circ}\times [G,G]})$ and $\check{\lambda}\in X_*(T)$ are dominant coweights.

\end{proof}

\begin{rem}

It is not difficult to see that \cite[Thm 16.3.3]{arinkin2020stack} and Theorem A also imply the equivalence (\ref{eq:0}). Indeed, the theorem in \emph{loc.cit} says that for reductive groups $G$ and $H$, the functor $$D_{\textrm{Nilp}}(\textrm{Bun}_G)\otimes D_{\textrm{Nilp}}(\textrm{Bun}_H)\rightarrow D_{\textrm{Nilp}}(\textrm{Bun}_G\times \textrm{Bun}_H)$$
is an equivalence where $\textrm{Nilp}$ denotes the global nilpotent cone (see \emph{loc.cit}).

\end{rem}
\newpage

\printbibliography

@article{arinkin2015singular,
  title={Singular support of coherent sheaves and the geometric Langlands conjecture},
  author={Arinkin, Dima and Gaitsgory, Dennis},
  journal={Selecta Mathematica},
  volume={21},
  number={1},
  pages={1--199},
  year={2015},
  publisher={Springer}
}

@article{arinkin2020stack,
  title={The stack of local systems with restricted variation and geometric Langlands theory with nilpotent singular support},
  author={Arinkin, Dima and Gaitsgory, Dennis and Kazhdan, David and Raskin, Sam and Rozenblyum, Nick and Varshavsky, Yasha},
  journal={arXiv preprint arXiv:2010.01906},
  year={2020}
}

@article{gaitsgory2013contractibility,
  title={Contractibility of the space of rational maps},
  author={Gaitsgory, Dennis},
  journal={Inventiones mathematicae},
  volume={191},
  number={1},
  pages={91--196},
  year={2013},
  publisher={Springer}
}

@article{beraldo2021geometric,
  title={On the geometric Ramanujan conjecture},
  author={Beraldo, Dario},
  journal={arXiv preprint arXiv:2103.17211},
  year={2021}
}

@article{faergeman2022non,
  title={Non-vanishing of geometric Whittaker coefficients for reductive groups},
  author={F{\ae}rgeman, Joakim and Raskin, Sam},
  year={2022}
}

@article{faergeman2021arinkin,
  title={The Arinkin-Gaitsgory temperedness conjecture},
  author={Faergeman, Joakim and Raskin, Sam},
  journal={arXiv preprint arXiv:2108.02719},
  year={2021}
}

@article{bezrukavnikov2007equivariant,
  title={Equivariant Satake category and Kostant-Whittaker reduction},
  author={Bezrukavnikov, Roman and Finkelberg, Michael},
  journal={arXiv preprint arXiv:0707.3799},
  year={2007}
}

@article{beraldo2021tempered,
  title={Tempered D-modules and Borel--Moore homology vanishing},
  author={Beraldo, Dario},
  journal={Inventiones mathematicae},
  volume={225},
  number={2},
  pages={453--528},
  year={2021},
  publisher={Springer}
}

@article{gaitsgory2017study,
  title={A Study in Derived Algebraic Geometry: Volume II: Deformations, Lie Theory and Formal Geometry},
  author={Gaitsgory, Dennis and Rozenblyum, Nick},
  journal={Mathematical surveys and monographs},
  volume={221},
  year={2017}
}

@article{gaitsgory2011ind,
  title={Ind-coherent sheaves},
  author={Gaitsgory, Dennis},
  journal={arXiv preprint arXiv:1105.4857},
  year={2011}
}

@article{drinfeld1995b,
  title={$ B $-structures on $ G $-bundles and local triviality},
  author={Drinfeld, Vladimir G and Simpson, Carlos},
  journal={Mathematical Research Letters},
  volume={2},
  number={6},
  pages={823--829},
  year={1995},
  publisher={International Press of Boston}
}

@incollection{hoffmann2010moduli,
  title={On moduli stacks of G-bundles over a curve},
  author={Hoffmann, Norbert},
  booktitle={Affine Flag Manifolds and Principal Bundles},
  pages={155--163},
  year={2010},
  publisher={Springer}
}

@article{beraldo2020spectral,
  title={The spectral gluing theorem revisited},
  author={Beraldo, Dario},
  journal={{\'E}pijournal de G{\'e}om{\'e}trie Alg{\'e}brique},
  volume={4},
  year={2020},
  publisher={Episciences. org}
}

@article{nadler2019spectral,
  title={Spectral action in Betti geometric Langlands},
  author={Nadler, David and Yun, Zhiwei},
  journal={Israel Journal of Mathematics},
  volume={232},
  number={1},
  pages={299--349},
  year={2019},
  publisher={Springer}
}

@article{drinfeld2013compact,
  title={Compact generation of the category of D-modules on the stack of G-bundles on a curve},
  author={Drinfeld, Vladimir and Gaitsgory, Dennis},
  year={2013}
}

@article{noohi2005foundations,
  title={Foundations of topological stacks I},
  author={Noohi, Behrang},
  journal={arXiv preprint math/0503247},
  year={2005}
}

@article{biswas2021fundamental,
  title={Fundamental groups of moduli of principal bundles on curves},
  author={Biswas, Indranil and Mukhopadhyay, Swarnava and Paul, Arjun},
  journal={Geometriae Dedicata},
  volume={214},
  number={1},
  pages={629--650},
  year={2021},
  publisher={Springer}
}

@article{drinfeld2013some,
  title={On some finiteness questions for algebraic stacks},
  author={Drinfeld, Vladimir and Gaitsgory, Dennis},
  journal={Geometric and Functional Analysis},
  volume={23},
  number={1},
  pages={149--294},
  year={2013},
  publisher={Springer}
}

@article{gaitsgory2022toy,
  title={A toy model for the Drinfeld--Lafforgue shtuka construction},
  author={Gaitsgory, Dennis and Kazhdan, David and Rozenblyum, Nick and Varshavsky, Yakov},
  journal={Indagationes Mathematicae},
  volume={33},
  number={1},
  pages={39--189},
  year={2022},
  publisher={Elsevier}
}

@book{kashiwara2016regular,
  title={Regular and irregular holonomic D-modules},
  author={Kashiwara, Masaki and Schapira, Pierre},
  volume={433},
  year={2016},
  publisher={Cambridge University Press}
}

@article{lysenko,
  title={Fourier coefficients and a filtration on $\mathrm{Shv}(\mathrm{Bun}_G)$},
  author={Lysenko, Sergey},
  journal={arXiv preprint math/2208.13500},
  year={2022}
}

@article{beraldo2017loop,
  title={Loop group actions on categories and Whittaker invariants},
  author={Beraldo, Dario},
  journal={Advances in Mathematics},
  volume={322},
  pages={565--636},
  year={2017},
  publisher={Elsevier}
}

@article{jiang2014automorphic,
  title={Automorphic integral transforms for classical groups I: endoscopy correspondences},
  author={Jiang, Dihua},
  journal={Automorphic forms and related geometry: assessing the legacy of II Piatetski-Shapiro},
  volume={614},
  pages={179--242},
  year={2014},
  publisher={American Mathematical Society Providence, RI}
}

\end{document}